\documentclass[12pt, american,reqno]{amsart}
\usepackage{amssymb}
\usepackage{graphicx}
\usepackage{amssymb}
 \usepackage{amsthm}
\usepackage{epsf}
\usepackage{amsbsy,amsmath}
\usepackage{mathtools}
\usepackage{mathrsfs}
\usepackage{amsfonts}
\usepackage{amssymb}
\usepackage{enumitem}
\usepackage{eucal}
\usepackage{graphics,mathrsfs}
\usepackage{amsthm}
\usepackage{secdot}
\usepackage{esint}
\usepackage{varwidth}
\usepackage{tasks}
\usepackage{bbm}
\usepackage{cite}
\theoremstyle{plain}
\newtheorem{theorem}{Theorem}[section]
\newtheorem{cor}{Corollary}[section]
\newtheorem{lemma}{Lemma}[section]

\theoremstyle{definition}

\newtheorem{remark}{Remark}

\allowdisplaybreaks

\newcommand{\G}{\mathbb G}
\newcommand{\Om}{\Omega}
\usepackage{xcolor}

\long\def\symbolfootnote[#1]#2{\begingroup
\def\thefootnote{\fnsymbol{footnote}}\footnote[#1]{#2}\endgroup}
\usepackage{xcolor}
\usepackage[colorlinks,citecolor=blue,hypertexnames=false]{hyperref}

\numberwithin{equation}{section}

\numberwithin{equation}{section}

\begin{document}
\title[Subelliptic  nonlocal Brezis-Nirenberg problems] {Subelliptic  Nonlocal Brezis-Nirenberg Problems on Stratified Lie  Groups}
\author{Sekhar Ghosh, Vishvesh Kumar and Michael Ruzhansky}
\address[ Sekhar Ghosh]{Department of Mathematics, National Institute of Technology Calicut, Kozhikode, Kerala, India - 673601}
\email{sekharghosh1234@gmail.com / sekharghosh@nitc.ac.in}
\address[Vishvesh Kumar]{Department of Mathematics: Analysis, Logic and Discrete Mathematics, Ghent University, Ghent, Belgium}
\email{vishveshmishra@gmail.com / vishvesh.kumar@ugent.be}

\address[Michael Ruzhansky]{Department of Mathematics: Analysis, Logic and Discrete Mathematics, Ghent University, Ghent, Belgium\newline and \newline
School of Mathematical Sciences, Queen Marry University of London, United Kingdom}
\email{michael.ruzhansky@ugent.be}



 \vskip -1cm  \hrule \vskip 1cm \vspace{-8pt}

\begin{abstract} 
In this paper, we investigate the subelliptic nonlocal  Brezis-Nirenberg problem on stratified Lie groups involving critical nonlinearities, namely,
\begin{align*}
    (-\Delta_{\mathbb{G}, p})^s u&= \mu |u|^{p_s^*-2}u+\lambda h(x, u) \quad \text{in}\quad \Omega, \\
u&=0\quad  \text{in}\quad \mathbb{G}\backslash \Omega,
\end{align*}
where $(-\Delta_{\mathbb{G}, p})^s$ is the fractional $p$-sub-Laplacian on a stratified Lie group $\mathbb{G}$ with homogeneous dimension $Q,$ $\Omega$ is an open bounded subset of $\mathbb{G},$ $s \in (0,1)$,  $\frac{Q}{s}>p\geq2,$ $p_s^*:=\frac{pQ}{Q-ps}$ is subelliptic fractional Sobolev critical exponent,  $\mu, \lambda>0$ are real parameters and $h$ is  a lower order perturbation of the critical power $|u|^{p_s^*-2}u$. Utilising direct methods of the calculus of variation, we establish the existence of at least one weak solution for the above problem under the condition that the real parameter $\lambda$ is sufficiently small. Additionally, we examine the problem for $\mu = 0$, representing subelliptic nonlocal equations on stratified Lie groups depending on one real positive parameter and involving a subcritical nonlinearity. We demonstrate the existence of at least one solution in this scenario as well. We emphasize that the results obtained here are also novel for $p=2$ even for the Heisenberg group. 
\end{abstract}
\dedicatory{ {To the memory of Professor Ha\"im Brezis, 1944--2024.}}
 \maketitle 
{ \textwidth=4cm \hrule}
\tableofcontents

\section{Introduction}
\setcounter{theorem}{0}\setcounter{lemma}{0}\setcounter{definition}{0}\setcounter{proposition}{0}\setcounter{remark}{0}

	 Recently, significant attention has been directed towards nonlocal subelliptic operators, including the fractional sub-Laplacian, on the Heisenberg group and, more generally, on stratified Lie groups. These operators are noteworthy for their intriguing theoretical structures and practical applications. It is practically impossible to provide a comprehensive list of references, but we refer to \cite{FF15, FF2015, PP22, GKR22, KD20, GKR23, MMPP23, GT21, GLV, K20, WD20, FMPPS18, FMMT15, WN19, RT16, FZ24, PP24, PP22, Pic22, FMMT15, RT20} and references therein.  
  Motivated by the above studies, in this paper we consider the following nonlinear subelliptic nonlocal equation involving critical nonlinearities: 
  \begin{equation} \label{pro1intro}
      \begin{cases} 
(-\Delta_{\mathbb{G}, p})^s u= \mu |u|^{p_s^*-2}u+\lambda h(x, u) \quad &\text{in}\quad \Omega, \\
u=0\quad & \text{in}\quad \mathbb{G}\backslash \Omega,
\end{cases}
  \end{equation}
where $(-\Delta_{\mathbb{G}, p})^s$ is the fractional $p$-sub-Laplacian on the stratified Lie group $\mathbb{G}$ with homogeneous dimension $Q,$ $\Omega$ is a open bounded subset of $\mathbb{G},$ $s \in (0,1)$,  $\frac{Q}{s}>p\geq2,$ $p_s^*:=\frac{pQ}{Q-ps}$ is subelliptic fractional Sobolev critical exponent,  $\mu \geq 0, \lambda>0$ are real parameters and $h$ is a subcritical nonlinearity satisfying \begin{equation} \label{growthintro}
        |h(x, t)| \leq a_1+a_2|t|^{q-1}\quad \text{for}\,\,\text{a.e.}\,\,x\in \Omega,\,\, \text{and for all } t \in \mathbb{R},
    \end{equation}
    for some $a_1, a_2>0$ and $q \in [1, p_s^*).$ The operator $(-\Delta_{\mathbb{G}, p})^s$ is defined as 
    \begin{equation}
        (-\Delta_{\mathbb{G}, p})^su(x):= C(Q,s, p)\,\,  P.V. \int_{{\mathbb{G}} } \frac{|u(x)-u(y)|^{p-2}(u(x)-u(y))}{\left|y^{-1} x\right|^{Q+p s}} d y, \quad x \in {\mathbb{G}},
    \end{equation}
    with $|\cdot|$ being a homogenous norm on the stratified Lie group $\G,$ $C(Q,s, p)$ is the positive normalization constant depending only on the homogeneous dimension $Q$ of $\G$, $s \in (0, 1)$ and $p \in (1, \infty).$ Here, the symbol $P. V.$ stands for the Cauchy principal value. Due to the nature of problem \eqref{pro1intro}, we will analyse it in two separate cases: first, when $\mu>0$, and second, when $\mu=0$. The case $\mu>0$ is generally more intricate due to the presence of the critical exponent in the nonlinearity and the absence of compactness phenomena of Sobolev embedding for the fractional $p$-sub-Laplacian on stratified Lie groups. Problems of the type \eqref{pro1intro} are commonly referred to in the literature as Brezis-Nirenberg type problems, stemming from their origin in the seminal paper of Brezis and Nirenberg \cite{BN83} for the Laplace operator on $\mathbb{R}^n$ in 1983.
    
    Let us briefly review some studies on the Brezis-Nirenberg problem in the Euclidean setting.  Brezis and Nirenberg \cite{BN83} considered  the following nonlinear elliptic partial differential critical equation:
\begin{equation} \label{EucBN}
    \begin{cases}
        &-\Delta u = |u|^{2^*-2} u+\lambda u \quad \text{in} \,\, \Omega,\\&
        u=0 \quad \text{on} \,\, \partial \Omega
    \end{cases} 
\end{equation} for $n\geq 3$ and $2^*:= \frac{2n}{n-2},$ where $\Omega$ is a smooth bounded domain of $\mathbb{R}^n$ and $\lambda$ is a real parameter. They established that if $n \geq 4$, then problem \eqref{EucBN} admits a positive solution for $\lambda \in (0, \lambda_1)$, where $\lambda_1$ denotes the first eigenvalue of $(-\Delta)$ in $H^1_0(\Omega)$. In the three-dimensional Euclidean space, they infer the existence of a constant $\lambda_* \in (0, \lambda_1)$ such that for any $\lambda \in (\lambda_*, \lambda_1)$, problem \eqref{EucBN} possesses a positive solution. Furthermore, they demonstrate that equation \eqref{EucBN} has a positive solution if and only if $\lambda \in (\lambda_1/4, \lambda_1)$ when $\Omega$ is a ball. Conversely, for $\lambda \notin (0, \lambda_1)$, it can be observed, using the Poho$\check{\text{z}}$aev identity, that problem \eqref{EucBN} has no positive solution. Capozzi et al. \cite{CFP85} established that for $n \geq 4$, equation \eqref{EucBN} admits a nontrivial solution for every parameter $\lambda$.

The primary motivation for investigating the Brezis-Nirenberg problem \eqref{EucBN}  arises from its resemblance to variational problems in differential geometry and physics. Notably, the celebrated {\it Yamabe problem} on Riemannian manifolds is a well-known example that is directly related to the Brezis-Nirenberg problem \eqref{EucBN}. Specifically, 

{\it Given an $n$-dimensional compact Riemannian manifold $(M, g)$, where $n \geq 3$, with scalar curvature $k:=k(x)$, the Yamabe problem seeks to find a metric $\tilde{g}$ conformal to $g$ with constant scalar curvature $\tilde{k}$.}

In fact, by setting $\tilde{g}= u^{\frac{4}{n-2}} g,$ where $u>0$ is the conformal factor, we can formulate the Yamabe problem in the form of the following nonlinear critical equation involving the Laplace-Beltrami operator $\Delta_M$ on $(M, g):$ 
\begin{equation}
    -\frac{n-1}{n-2} \Delta_M u = \tilde{k} u^{2^*-1}- k(x) u.
\end{equation}
Additionally, the Brezis-Nirenberg problem \eqref{EucBN} is linked to the existence of extremal functions for functional inequalities and the existence of non-minimal solutions for Yang-Mills functions and $H$-system (see \cite{BN83}). Problem \eqref{EucBN} has been extensively studied in many different settings involving different local and nonlocal operators, e.g. $p$-Laplacian and fractional $p$-Laplacian.  For instance, in \cite{PP14}, Palatucci and Pisante obtained the concentration-compactness principle on a bounded domain for the fractional Laplacian and as an application, the existence of a weak solution is established to a subcritical problem. The main tools used are an improved Sobolev inequality combined with the Caffarelli-Silvestre extension.  Mosconi et al. \cite{MPSY16} studied the Brezis-Nirenberg problem for fractional $p$-Laplacian on a bounded domain. Here, the authors obtained the existence of a nontrivial solution to the critical Brezis-Nirenberg problem. The authors established the existence results with an appropriate upper bound for the mountain pass level for the corresponding energy functional along with the optimal decay for the extremals of the Sobolev constant. In \cite{MM17}, Mawhin and Molica Bisci studied the problem \eqref{pro1intro} in the context of Euclidean space $\mathbb{R}^n$ associated with the fractional $p$-Laplacian. They proved the existence of at least one weak solution, provided that the parameter $\lambda$ is sufficiently small. The results obtained here are routed from direct methods from the calculus of variations without using the concentration-compactness principle of Lions \cite{Lions85}. The Brezis-Nirenberg problem for the fractional $p$-Laplacian in the entire space $\mathbb{R}^n$ has been investigated in \cite{BSS18}, employing the concentration-compactness principle of Lions \cite{Lions85} and precise knowledge of the decay of optimal $p$-Sobolev extremizers. For a concise development on nonlocal Brezis-Nirenberg problems, we refer  \cite{BS15, MRS16, DNK23, GLPY21, HBV16, MM17, SER13, SER14, SERV05, SERV13, Shen24, WS22} and the references therein.

 Now, let us discuss the motivation for considering the Brezis-Nirenberg type problem on stratified Lie groups. 
In the realm of CR geometry, the {\it CR Yamabe problem} was first explored in the seminal work by Jerison and Lee \cite{JL}.  The CR Yamabe problem is: 

{\it Given a compact, strictly pseudoconvex CR manifold, find a choice of contact form for which
the pseudohermitian scalar curvature is constant.} 

The Heisenberg group, the simplest example of a stratified Lie group, plays a similar role in CR geometry as the Euclidean space in conformal geometry. Consequently, the analysis on stratified Lie groups has proven to be a fundamental tool in resolving the CR Yamabe problem. Consequently, substantial interest has been devoted to studying subelliptic PDEs on stratified Lie groups. As one can also expect, this is also related to determining the extremals for the Sobolev inequality on the Heisenberg group due to Folland and Stein \cite{FS82} and this has been investigated by Jerison and Lee \cite{JL89}. We also refer to \cite{RTY20} for a description of best constants of higher order Sobolev inequalities associated with a Rockland operator on graded Lie group, and to  \cite{GKR23} for best constants of fractional Sobolev inequalities on stratified Lie groups and related subelliptic variational problems. In recent times, researchers have delved into the fractional CR Yamabe problem and related problems \cite{GQ13, CW17, FMMT15, KMW17, K20, CHY21, DGN07}.

On another note, the importance of nilpotent Lie groups in deriving sharp subelliptic estimates for differential operators on manifolds was highlighted in the seminal paper by Rothschild and Stein \cite{RS76}. The Rothschild-Stein lifting theorem demonstrates that a general H\"ormander's sums of squares of vector fields on manifolds can be approximated by a sup-Laplacian on some stratified Lie group (see also \cite{F77} and \cite{Roth83}). As a result, it becomes crucial to study partial differential equations on stratified Lie groups, leading to numerous interesting and promising works that merge Lie group theory with the analysis of partial differential equations. In recent decades, there has been a rapid growth of interest in sub-Laplacians on stratified Lie groups due to their relevance not only in theoretical settings but also in practical applications, such as mathematical models of crystal materials and human vision (see \cite{C98} and \cite{CGS04}).

For the sub-Laplacian on stratified Lie group, the Brezis-Nirenberg problem 
\begin{equation} \label{BNLocal}
    \begin{cases}
        -\Delta_{\G} u=|u|^{2^*-2} u+\lambda u \quad &\text{in} \quad \Omega, \\ u=0 \quad &\text{in} \quad \partial\Omega,
         \end{cases}
\end{equation}
where $2^*=2Q/(Q-2)$ is the Folland-Sobolev critical exponent, was studied by  Loiudice \cite{L07} extending the Euclidean results in \cite{BN83}. She proved that \eqref{BNLocal} has at least a positive solution in $S^1_0(\Omega)$  for any $0<\lambda<\lambda_1,$ where $\lambda_1$ is the first eigenvalue of $-\Delta_{\G}$ in $S^1_0(\Omega).$ It is also known that for $\lambda \geq \lambda_1,$ there is no positive solution of \eqref{BNLocal}. The problem \eqref{BNLocal} was studied by Citti \cite{Citti95} in the setting of the Heisenberg group using the explicit form of minimizers of subelliptic Sobolev inequality due to Jerison and Lee \cite{JL89}. However, in the context of Carnot groups, the explicit form of the Sobolev extremals is not known, posing a challenge for applying the Brezis-Nirenberg approach. To address this, Loiudice \cite{L07} utilised the exact behaviour of any extremal function of the Folland-Stein Sobolev inequality by means of the deep analysis performed by Bonfiglioli and Uguzzoni \cite{BU04}. Here, it is worth noting that Garofalo-Vassilev \cite{GV00} proved that the best constant in the Folland–Stein embedding on $\G$ is attained at some function using the concentration-compactness principle of Lions \cite{Lions85}. Indeed, the concentration-compactness (CC) principle is one of the most extensively used tools in addressing the lack of compactness. It allows for proving the existence of a minimum for corresponding energy functional as the weak limit of a minimizing sequence.   In 2017, Molica Bisci and Repov$\check{\text{s}}$ \cite{BR17} investigated  \eqref{BNLocal} using direct variational methods and proved the existence of one weak solution in the Folland-Stein space $S^1_0(\Omega)$ without any use of concentration-compactness techniques, provided that the
parameter $\lambda$ is sufficiently small. They employed a weakly lower semicontinuity technique previously used in the literature for studying quasilinear equations involving critical nonlinearities in the Euclidean setting, as seen in \cite{FF15, Squ04}. We refer to several other interesting papers \cite{Lan03, GL92, BPV22, MBF16, MM07, BFP20a} and references therein dealing with subelliptic nonlinear equations on stratified Lie groups. 

The problem \eqref{BNLocal} was subsequently extended to the $p$-sub-Lapalcian on stratified Lie group by Loiudice  \cite{Loi22} using the knowledge of the exact rate of decay of the $p$-Sobolev extremals on stratified Lie groups investigated in  \cite{L19}. Specifically, she proved that the problem \begin{equation} \label{BNpLocal}
    \begin{cases}
        -\Delta_{p, \G} u=|u|^{p^*-p} u+\lambda |u|^{p-2} u \quad &\text{in} \quad \Omega, \\ u=0 \quad &\text{in} \quad \partial\Omega,
         \end{cases}
\end{equation} admits a positive solution for any $\lambda \in (0, \lambda_1)$ for $1<p<Q,$ where $\lambda_1$ is the first eigenvalue of $-\Delta_{p, \G}$ in $\Omega.$ On the other hand, Pucci and Temperini \cite{PT22} studied the entire solutions of quasi-linear equation involving $p$-sub-Laplacian on the Heisenberg group $\mathbb{H}^n.$ In \cite{H20}, the author considered the integral type Brezis-Nirenberg problems on the Heisenberg group. Utilizing the sharp Hardy-Littlewood-Sobolev inequalities on the Heisenberg group \cite{FL}, the author obtained nonexistence and existence results for the problem. 

Recently, there have been plenty of research activities dealing with  a general class on the nonlinear and nonlocal operators on the Heisenberg group or in general stratified Lie groups, namely, the fractional $p$-sub-Laplacian, $(-\Delta_{\mathbb{G}, p})^s$ defined by 
\begin{equation}
        (-\Delta_{\mathbb{G}, p})^su(x):= C(Q, s, p)  P.V. \int_{{\mathbb{G}} } \frac{|u(x)-u(y)|^{p-2}(u(x)-u(y))}{\left|y^{-1} x\right|^{Q+p s}} d y, \quad x \in {\mathbb{G}}.
    \end{equation}
We refer to a non-exhausted list of papers \cite{FF15, FF2015, PP22, GKR22, KD20, GKR23, MMPP23, FZ24, GT21, GLV, K20, WD20, FMPPS18, FMMT15, WN19, RT16,RT20, PP24, PP22, Pic22} and references therein for recent developments in this context. For $p=2$, the operator corresponds to the fractional sub-Laplacian $(-\Delta_\G)^s$, which has been extensively studied and found connections with various areas of mathematics in several intriguing papers \cite{BF13,RT16,FMMT15,GLV,FMPPS18,K20,RT20}, along with references cited therein. This operator, initially defined for $u \in C_c^\infty(\G)$, is given by:
\begin{equation}\label{fracl}
(-\Delta_\G)^su(x):= \lim_{\epsilon \rightarrow 0} \int_{\G \backslash B(x, \epsilon)} \frac{|u(x)-u(y)|}{|y^{-1}x|^{Q+2s}} d y = C(Q, s)\,\, P.V. \int_{\G} \frac{|u(x)-u(y)|}{|y^{-1}x|^{Q+2s}} d y.
\end{equation}
It is known that for an $H$-type group, the operator $(-\Delta_\G)^s,$ $s \in (0, \frac{1}{2}),$ is a multiple of the pseudo-differential operator defined as:
\begin{equation}
\mathcal{L}_s:=2^s (-\Delta_z)^{\frac{s}{2}} \frac{\Gamma (-\frac{1}{2}\Delta_\G (-\Delta_z)^{-\frac{1}{2}}+\frac{1+s}{2}) }{\Gamma (-\frac{1}{2}\Delta_\G (-\Delta_z)^{-\frac{1}{2}}+\frac{1-s}{2})},
\end{equation}
where $-\Delta_z$ is the positive Laplacian in the center of the $H$-type group $\G$, and $-\Delta_\G$ is the sub-Laplacian on $\G$. For further details, we refer to \cite{BF13,RT16,RT20}. It is noteworthy that $\mathcal{L}_s$ is a ``conformal invariant" operator and holds significance in CR geometry (see \cite{FMMT15}). Finally, we highlight that the operator $(-\Delta_\G)^s$ does not coincide with the standard fractional power $-\Delta_{\G}^s$ of the sublaplacian $-\Delta_{\G}$ in the Heisenberg group or in general, $H$-type groups, for any value of $s \in (0, 1)$, which is defined as:
$$(- \Delta_{\G}^su)(x):=-\frac{s}{\Gamma (1-s)}\int_0^\infty \frac{1}{t^{1+s}}(H_tu(x)-u(x))\,d t, $$
where $H_t:=e^{-t \Delta_\G}$ is the heat semigroup constructed by Folland \cite{F75}. It is also known that $$\lim_{s \rightarrow 1^-} (-\Delta_\G)^s u = -\Delta_\G u $$ for all $u \in C_0^\infty(\G),$ see \cite[Proposition 1]{PP22}. In a recent paper, Garofalo and Tralli \cite{GT21} computed explicitly the fundamental solutions of nonlocal operators $(-\Delta_\G)^s$ and $(- \Delta_{\G}^su)$ and proved some intertwining formulas on $H$-type groups.  

 In \cite{GKR22}, we investigated an eigenvalue problem for $(-\Delta_{\mathbb{G}, p})^s$ on fractional Folland-Stein-Sobolev spaces $X_0^{s, p}(\G)$ on $\G$ using the variational method. The main challenge in dealing with the operator $(-\Delta_{\mathbb{G}, p})^s$ lies in its definition. We need to manage both its nonlocal and nonlinear nature, as well as the non-Euclidean geometry of stratified Lie groups. When $p\neq 2,$ many classical tools that have been adapted to the sub-Riemannian case in the linear setup, such as the Dirichlet-to-Neumann map in \cite{FMMT15, GT21} and the approach via the non-commutative Fourier representation, are not readily applicable. Dealing with nonlocal operators typically involves considering a tail-type contribution to precisely control the inherent long-range interactions. Consequently, in \cite{GKR22}, we introduced a nonlocal tail within the framework of stratified Lie groups. It is worth mentioning that a similar concept within the context of the Heisenberg group was independently introduced by Palatucci and Piccinini \cite{PP22}, around the same time. We refer to \cite{MMPP23, Pic22, GLV} for applications of the nonlocal tail concept to deal with different kinds of problems in the Heisenberg group as well as in the general context of homogeneous groups. For the study of the regularity theory of subelliptic nonlocal equations involving $(-\Delta_{\mathbb{G}, p})^s,$ one may consult recent articles \cite{MMPP23, PP22, Pic22, FZ24} and references therein.  In \cite{GKR22}, we also have established the compactness of fractional Folland-Stein-Sobolev embeddings on $\G$ which serve as a crucial tool to study nonlocal equations. For further exploration of the best constants of such fractional Folland-Stein-Sobolev inequalities, we recommend consulting \cite{GKR23}. 
 
 Continuing this line of research on nonlocal subelliptic PDEs, in this paper we consider the Brezis-Nirenberg problem
\eqref{pro1intro} associated with the fractional $p$-sub-Laplacian $(-\Delta_{\mathbb{G}, p})^s$ on $\G.$

We state the first main result of this paper. 

\begin{theorem} \label{thm1.1}
    Let $\Omega$ be a bounded open subset of a stratified Lie group $\G$ with the homogeneous dimension $Q.$ Let $s \in (0, 1),$ $\frac{Q}{s}> p\geq 2$ and $q \in [1,p_s^*),$ where $p_s^*=\frac{pQ}{Q-ps}.$ Let $h:\Omega \times \mathbb{R} \rightarrow \mathbb{R}$ be a Carath\'eodory function with respect to Haar measure $dx$ on $\G$, which means, $h$ is Haar measurable in the first variable and continuous in the second variable,  such that 
    \begin{equation} \label{growth}
        |h(x, t)| \leq a_1+a_2|t|^{q-1}\quad \text{for}\,\,\text{a.e.}\,\,x\in \Omega,\,\, \forall t \in \mathbb{R},
    \end{equation}
    for some $a_1, a_2>0.$ Then, for every $\mu>0,$ there exists $\Lambda_\mu>0$ such that the following nonlocal problem
    \begin{equation} \label{pro1}
       \begin{cases} 
(-\Delta_{\mathbb{G}, p})^s u= \mu |u|^{p_s^*-2}u+\lambda h(x, u) \quad &\text{in}\quad \Omega \\
u=0\quad & \text{in}\quad \mathbb{G}\backslash \Omega
\end{cases},
    \end{equation} admits at least one (real-valued) weak solution $u_\lambda \in X^{s,p}_0(\Omega)$ for every $0<\lambda<\Lambda_\mu$ which is a local minimum of the energy functional 
    \begin{equation} \label{funcintro}
        \mathfrak{I}_{\mu, \lambda}(u):= \frac{1}{p} \iint_{\G \times \G} \frac{|u(x)-u(y)|^p}{|y^{-1} x|^{Q+sp}} d x d y -\frac{\mu}{p_s^*} \int_{\Omega} |u(x)|^{p_s^*} d x-\lambda \int_{\Omega} \int_0^{u(x)} h(x, \tau)d \tau,
    \end{equation} for every $u \in X^{s,p}_0(\Omega).$
\end{theorem}
We will also provide explicit information on the interval $(0, \Lambda_\mu).$ In fact, we show that 
$$(0, \Lambda_\mu) \subset \left(0, \max_{r \geq 0} \frac{r^{p-1}-\mu C_{Q, s,p^*_s, \Om}^{p_s^*} r^{p_s^*-1} }{a_1 C_{Q, s,p^*_s, \Om} |\Omega|^{\frac{p_s^*-1}{p_s^*}}+a_2C_{Q, s,p^*_s, \Om}^q|\Omega|^{\frac{p_s^*-q}{p_s^*}} r^{q-1}}\right),$$ where 
 $C_{Q, s,p^*_s, \Om}$ is the best constant in the fractional subelliptic Sobolev inequality given by \eqref{bestcon} and $|\Omega|$ is the Haar measure of the open bounded subset $\Omega$ of $ \G$.
\\

The proof of Theorem \ref{thm1.1} employs direct variational methods and circumvents the use of concentration-compactness techniques, given that the parameter $\lambda$ is sufficiently small while $\mu$ can be any real value. The primary challenge in proving this theorem is that the functional $\mathfrak{I}_{\mu, \lambda}:X_0^{s, p}(\Omega) \rightarrow \mathbb{R}$ defined by \eqref{funcintro} does not globally satisfy the ``Palais-Smale" condition. We will first prove a weak lower semicontinuity result in  Lemma \ref{weaksemiconti} to overcome this issue. Indeed, we will show that the restriction of the functional
$$u \mapsto \frac{1}{p} \iint_{\G \times \G} \frac{|u(x)-u(y)|^p}{|y^{-1} x|^{Q+sp}} d x d y -\frac{\mu}{p_s^*} \int_{\Omega} |u(x)|^{p_s^*} d x$$
to $\overline{B_{X_0^{s,p}(\Omega)}(0, r_{\mu})}$ for sufficiently small $r_{\mu}$ is sequentially weakly lower semicontinuous. This technique has already been used in studying local and nonlocal equations \cite{FF15, FF2015, FF2015, BR17, MM17}.  As a result, the energy functional $\mathfrak{I}_{\mu, \lambda}$ associated with \eqref{pro1} is locally sequentially weakly lower semicontinuous. Through direct minimization, we can demonstrate that for any $\mu>0$ and sufficiently small $\lambda$, the functional $\mathfrak{I}_{\mu, \lambda}$ has a critical point (local minimum), which serves as a weak solution to problem \eqref{pro1}.
It is important to note that the authors in \cite{BR17} used truncation arguments to prove the weak lower semicontinuity result. However, this approach is less effective in the nonlocal setting. Furthermore, extending the range from $p \in [2, \infty)$ to $p \in (1, \infty)$ is challenging due to a technical result used in the proof of Lemma \ref{weaksemiconti}. A different method will likely be needed to handle this broader range, which we plan to investigate in the future.

In the next result, we will study a specific case of the problem \eqref{pro1intro}, namely, equation \eqref{pro15.7intro} below. In this scenario, we will deduce the existence of a non-negative nontrivial weak solution.  

\begin{theorem} \label{thmnonmain}
    Let $\Omega$ be a bounded open subset of a stratified Lie group $\G$ with the
homogeneous dimension $Q.$ Let  $s \in (0,1)$ and $p \in (1, \infty)$ such that $\frac{Q}{s}>p\geq 2.$ Suppose that $m$ and $q$ are two real constants such that 
$$1 \leq m<p\leq q<p_s^*=\frac{pQ}{Q-ps}.$$
Then, there exists an open interval $\Lambda \subset (0, +\infty)$ such that, for every $\lambda \in \Lambda,$ the subelliptic nonlocal problem 
\begin{equation} \label{pro15.7intro}
       \begin{cases} 
(-\Delta_{\mathbb{G}, p})^s u=  |u|^{p_s^*-2}u+\lambda (|u|^{m-1}+|u|^{q-1}) \quad &\text{in}\quad \Omega, \\
u=0\quad & \text{in}\quad \mathbb{G}\backslash \Omega,
\end{cases}
    \end{equation} has at least a nonnegative nontrivial weak solution $u_\lambda \in X_{0}^{s,p}(\Omega).$ 
\end{theorem}

In Theorem \ref{thmnonmain}, the interval $\Lambda$ can be explicitly localised and given by  
\begin{equation}
  \Lambda \subset \left( 0,\,  \max_{r>0} \frac{r^{p-1}- C_{Q, s,p^*_s, \Om}^{p_s^*} r^{p_s^*-1} }{ C_{Q, s,p^*_s, \Om} |\Omega|^{\frac{p_s^*-1}{p_s^*}}+C_{Q, s,p^*_s, \Om}^m|\Omega|^{\frac{p_s^*-m}{p_s^*}} r^{m-1}+C_{Q, s,p^*_s, \Om}^q|\Omega|^{\frac{p_s^*-q}{p_s^*}} r^{q-1}} \right).
\end{equation}

Now, we will focus on the case when $\mu=0$ in problem \eqref{pro1intro}. In \cite{FBR17}, Ferrara, Bisci and Repov$\check{\text{s}}$  considered such a problem for the sub-Laplacian on stratified Lie groups and proved the existence of a weak solution. We aim to extend  \cite{FBR17} for nonlinear and nonlocal operators, namely for the fractional $p$-sub-Laplacian on stratified Lie groups. Additionally, we point to \cite{MBF16}, where the existence of multiple solutions for parametric elliptic equations on Carnot groups has been demonstrated. This was achieved by leveraging the renowned Ambrosetti-Rabinowitz condition and a local minimum result by Ricceri \cite{Ric00}.
 We state our result in the form of the following theorem:

 \begin{theorem} \label{thm1.2intro}
    Let $\Omega$ be a bounded open subset of a stratified Lie group $\G$ with the homogeneous dimension $Q.$ Let $s \in (0, 1),$ $1<p<q<p_s^*=\frac{pQ}{Q-ps}$ and $\frac{Q}{s}>p\geq2.$ Let $h:\Omega \times \mathbb{R} \rightarrow \mathbb{R}$ be a Carath\'eodory function such that 
    \begin{equation} 
        |h(x, t)| \leq a_1+a_2|t|^{q-1}\quad \text{for}\,\,\text{a.e.}\,\,x\in \Omega,\,\, \forall t \in \mathbb{R},
    \end{equation}
    for some $a_1, a_2>0.$ Furthermore, let 
    $$0<\lambda < \frac{(q-p)^{\frac{q-p}{q-1}} (p-1)^{\frac{p-1}{q-1}}}{(a_1 C_1)^{\frac{q-p}{q-1}} (a_2C_2)^{\frac{p-1}{q-1}} (q-1) C_{Q, s, p_s^*, \Omega}^p |\Omega|^{\frac{p_s^*-q}{p_s^*}\left( \frac{p-1}{q-1} \right)} |\Omega|^{\frac{p_s^*-1}{p_s^*} \left( \frac{q-p}{q-1} \right)}  },$$ where $C_1,$ $C_2$ and  $C_{Q, s, p_s^* , \Omega}$ is the embedding constant of $L^{p_s^*}(\Omega) \hookrightarrow L^1(\Omega),$ $L^{p_s^*}(\Omega) \hookrightarrow L^q(\Omega),$ and  $X_{0}^{s, p}(\Omega) \hookrightarrow L^{p_s^*}(\Omega),$ respectively and $|\Omega|$ is the Haar measure of set $\Omega.$ Then, the following  nonlocal subelliptic parameter  problem
    \begin{equation} \label{pro2intro}
       \begin{cases} 
(-\Delta_{\mathbb{G}, p})^s u= \lambda h(x, u) \quad &\text{in}\quad \Omega \\
u=0\quad & \text{in}\quad \mathbb{G}\backslash \Omega
\end{cases},
    \end{equation} admits a weak solution  $u_{0, \lambda} \in X^{s,p}_0(\Omega)$ and 
    $$\|u_{0,  \lambda}\|_{X^{s,p}_0(\Omega)}< \left( \frac{\lambda (q-1) a_2 C_2 C_{Q,s,p_s^*, \Omega}^q |\Omega|^{\frac{p_s^*-q}{p_s^*}}}{p-1} \right)^{\frac{1}{p-q}}.$$ 
\end{theorem}

 This paper will be structured as follows: In the next section, we will introduce the fundamentals of analysis on stratified Lie groups and present the fractional Sobolev-Folland-Stein spaces, which are crucial for subsequent discussions. Section \ref{sec3} will focus on several supporting results, notably the locally sequentially weakly lower semicontinuous result, which will serve as the foundation for establishing the main results. The proof of Theorem \ref{thm1.1} will be provided in Section \ref{sec4}. In Section \ref{sec5}, we will present some applications and implications of Theorem \ref{thm1.1}. Finally, in the last section, we will address a subcritical parametric nonlocal subelliptic equation associated with the fractional $p$-sub-Laplacian and demonstrate the existence of a weak solution.

\section{Preliminaries: The Stratified Lie group and fractional Sobolev spaces}\label{s2}

In this section, we will review some essential tools related to stratified Lie groups and the fractional Sobolev spaces defined on them. There are multiple approaches to introducing the concept of stratified Lie groups, and interested readers can refer to various books and monographs such as \cite{FS82, BLU07, GL92, FR16, RS19,F75}.  
\subsection{Stratified Lie groups}
A Lie group $\mathbb{G}$ (on $\mathbb{R}^N$ with a Lie group law $\circ$) is said to be {\it homogeneous} if, for each $\lambda>0$, there exists an automorphism $D_{\lambda}:\mathbb{G} \rightarrow\mathbb{G}, $ called a  {\it dilation}, defined by $D_{\lambda}(x)=  (\lambda x^{(1)}, \lambda^2 x^{(2)},..., \lambda^{k}x^{(k)})$ for $x^{(i)} \in \mathbb{R}^{N_i},\,\forall\, i=1,2,..., k$ and $N_1 + N_2+ ... + N_k = N.$   We denote the Lie algebra associated with the Lie group $\G =(\mathbb{R}^N, \circ),$ that is, the Lie algebra of left-invariant (with respect to the group law $\circ$) vector fields on $\mathbb{R}^N$ by $\mathfrak{g}$. With $N_1$ the same as in the above decomposition of $\mathbb{R}^N$, let $X_1, ..., X_{N_1} \in \mathfrak{g} $ be such that $X_i(0) = \frac{\partial}{\partial x_i}|_0$ for $i = 1, ..., N_1$. We make the following assumption: 

{\it The H\"ormander condition $rank(Lie\{X_1, ..., X_{N_1} \}) = N$
		holds for every $x \in \mathbb{R}^{N}$, that is,   the Lie algebra generated  by  $X_1, ..., X_{N_1}$ is the whole $\mathfrak{g}$.}

With the above hypothesis, we call $\G =(\mathbb{R}^n, \circ, D_\lambda)$ a stratified Lie group (or a homogeneous Carnot group). Here $k$ is called the step of the stratified Lie group and $n:=N_1$ is the number of generators. The number $Q:=\sum_{i=1}^{k}iN_i$ is called the homogeneous dimension of $\G.$ At times, we write $\lambda x$ to denote the dilation $D_{\lambda}x$.  The Haar measure on $\mathbb{G}$ is denoted by $d x$ and it is nothing but the usual Lebesgue measure on $\mathbb{R}^N.$   Let $\Omega$ be a Haar measurable subset of $\mathbb{G}$. Then $|D_{\lambda}(\Omega)|=\lambda^{Q}|\Omega|$ where $|\Omega|$ is the Haar measure of $\Omega$. 

 We would like to note here that in the literature, a {\it stratified} Lie group (or Carnot group) $\mathbb{G}$ is defined as a connected and simply connected Lie group whose Lie algebra $\mathfrak{g}$ is {\it stratifiable}. This means $\mathfrak{g}$ admits a vector space decomposition $\mathfrak{g} = \bigoplus_{i=1}^k \mathfrak{g}_i$ such that  $[\mathfrak{g}_1, \mathfrak{g}_i] = \mathfrak{g}_{i+1}$ for all $i=1.2,\ldots, k-1$ and  $[\mathfrak{g}_1, \mathfrak{g}_k] = \{0\}.$ 
 It is not difficult to see that any homogeneous Carnot group is a Carnot group according to the classical definition (see \cite[Theorem 2.2.17]{BLU07}). The opposite implication is also true; we refer to \cite[Theorem 2.2.18]{BLU07} for a detailed proof.  

The advantage of our operative definition of a stratified Lie group is that it is not only more convenient for analytic purposes but also enables us to gain a clearer understanding of the underlying group law.   Examples of stratified Lie groups include the Heisenberg group, more generally, $H$-type groups, the Engel group, and the Cartan group.  

For any $x,y\in\mathbb{G}$ the Carnot-Carath\'{e}odory distance is defined as
$$\rho_{cc}(x,y)=\inf\{l>0: ~\text{there exists an admissible}~ \gamma:[0,l]\rightarrow\mathbb{G} ~\text{with}~ \gamma(0)=x ~\text{\&}~ \gamma(l)=y \}.$$
We define $\rho_{cc}(x,y)=0$, if no such curve exists. We recall that an absolutely continuous curve $\gamma:[0,1]\rightarrow\mathbb{R}$ is said to be admissible, if there exist functions $c_i:[0,1]:\rightarrow\mathbb{R}$, for $i=1,2...,n$ such that
$${\dot{\gamma}(t)}=\sum_{i=1}^{n}c_i(t)X_i(\gamma(t))~\text{and}~ \sum_{i=1}^{n}c_i(t)^2\leq1.$$
Note that the functions $c_i$ may not be unique since the vector fields $X_i$ might not be linearly independent. Generally,  $\rho_{cc}$ is not a metric, but the H\"ormander condition for the vector fields $X_1,X_2,..., X_{n}$ guarantees that $\rho_{cc}$ is indeed a metric. Consequently, the space $(\mathbb{G}, \rho_{cc})$ is classified as a Carnot-Carathéodory space.

	A continuous function $|\cdot|: \mathbb{G} \rightarrow \mathbb{R}^{+}$ is said to be a homogeneous quasi-norm on a homogeneous Lie group $\mathbb{G}$ if it is symmetric ($|x^{-1}| = |x|$ for all $x \in\mathbb{G}$), definite ($|x| = 0$ if and only if $x = 0$), and $1$-homogeneous ( $|D_\lambda x| = \lambda |x|$ for all $x \in\mathbb{G}$ and $\lambda>0$).

An example of a quasi-norm on \(\mathbb{G}\) is the norm defined as \(d(x):=\rho_{cc}(x, 0),\,\, x\in \mathbb{G}\), where \(\rho\) is a Carnot-Carathéodory distance associated with H\"ormander vector fields on \(\mathbb{G}\). It is known that all homogeneous quasi-norms on \(\mathbb{G}\) are equivalent. In this paper, we will use a left-invariant homogeneous distance \(d(x, y):=|y^{-1} \circ x|\) for all \(x, y \in \mathbb{G}\), which is derived from the homogeneous quasi-norm on \(\mathbb{G}\). The quasi-ball of radius $r$ centered at $x\in\mathbb{G}$ with respect to the quasi-norm $|\cdot|$ is defined as
\begin{equation}\label{d-ball}
	B(x, r)=\left\{y \in \mathbb{G}: \left|y^{-1} \circ x\right|<r\right\}.
\end{equation}

As a consequence of the H\"ormander hypoelliptic condition,
the sub-Laplacian (or horizontal Laplacian)  on $\mathbb{G}$ defined as
\begin{equation}\label{d-sub-lap}
	\mathcal{L}:=X_{1}^{2}+\cdots+X_{n}^{2}
\end{equation}
is a hypoelliptic operator. 
The horizontal gradient on $\mathbb{G}$ is defined as
\begin{equation}\label{d-h-grad}
	\nabla_{\mathbb{G}}:=\left(X_{1}, X_{2}, \cdots, X_{n}\right).
\end{equation}
For $p \in(1,+\infty)$, we define the $p$-sub-Laplacian on the stratified Lie group $\mathbb{G}$ as
\begin{equation}\label{d-p-sub}
	\Delta_{\mathbb{G},p} u:=\operatorname{div}_{\mathbb{G}}\left(\left|\nabla_{\mathbb{G}} u\right|^{p-2} \nabla_{\mathbb{G}} u\right).
\end{equation}
For $s \in (0, 1)$ and $p \in (1, \infty),$ we define the fractional $p$-sub-Laplacian, $(-\Delta_{\mathbb{G}, p})^s$ as 
\begin{equation}
        (-\Delta_{\mathbb{G}, p})^s(u)(x):= C(p, Q, s)  P.V. \int_{{\mathbb{G}} } \frac{|u(x)-u(y)|^{p-2}(u(x)-u(y))}{\left|y^{-1} x\right|^{Q+p s}} d y, \quad x \in {\mathbb{G}},
    \end{equation}
     where $|\cdot|$ is a homogeneous norm on the stratified Lie group $\G,$ $C(Q,s, p)$ is the positive normalization constant depending only on the homogeneous dimension $Q$, $s$ and $p.$



\subsection{Fractional Sobolev-Folland-Stein spaces on stratified Lie groups}

We are now in a position to define the notion of fractional Sobolev-Folland-Stein type spaces useful for our study. We refer to \cite{GKR22} for complete details and comparison with other definitions of fractional Sobolev spaces on $\G$.

Let $\G$ be a stratified Lie group.  Then for $0<s<1< p<\infty$, the fractional Sobolev space $W^{s,p}(\G)$ on stratified groups is defined as 

\begin{equation}
	W^{s,p}(\G):=\{u\in L^{p}(\G): [u]_{s, p,\Omega}<\infty\},
\end{equation}
endowed with the norm 
\begin{equation}
	\|u\|_{W^{s,p}(\G)}=\|u\|_{L^p(\G)}+[u]_{s,p,\G},
\end{equation}
where $[u]_{s, p,\G}$ denotes the Gagliardo semi-norm defined by
\begin{equation}
	[u]_{s, p,\G}:=\left(\int_{\G} \int_{\G} \frac{|u(x)-u(y)|^{p}}{\left|y^{-1} x\right|^{Q+ps}} d xd y\right)^{\frac{1}{p}}<\infty.
\end{equation} 
In a similar way, 
given an open subset  $\Omega \subset {\mathbb{G}},$ we  define the fractional Sobolev-Folland-Stein type $X_0^{s,p}(\Omega)$ as a closed subspace of $W^{s,p}(\G)$ by $$X_0^{s,p}(\Omega)=\{u\in W^{s,p}(\G): u=0 \text{ in }\G\setminus\Omega\},$$ equipped with the  norm $\|u\|_{X^{s,p}_0(\Omega)}:=\|u\|_{L^p(\Omega)}+[u]_{s, p,\mathbb{G}}$. When $\Omega$ is bounded, then by Poincar\'e inequality, $\|u\|_{X^{s,p}_0(\Omega)}:=[u]_{s, p,\mathbb{G}}$ becomes a norm on $X_0^{s,p}(\Omega)$, (see \cite{GKR22}). The space $X_0^{s,p}(\Omega)$ is a reflexive and separable Banach space for $1<p<\infty$ and $1\leq p<\infty$, respectively. For any $\varphi\in X_0^{s,p}(\Omega)$, we have
\begin{equation}
	\langle\left(-\Delta_{p,{\mathbb{G}}}\right)^s u,\varphi\rangle= \iint_{\mathbb{G} \times \mathbb{G}} \frac{|u(x)-u(y)|^{p-2}(u(x)-u(y))(\varphi(x)-\varphi(y))}{\left|y^{-1} x\right|^{Q+p s}} d xd y.
\end{equation}

The next result is about the inclusion of fractional Sobolev spaces  $X_0^{s, p}(\Omega)$ into the Lebesgue spaces. In particular, the compactness of the Sobolev embedding is highly useful for dealing with fractional nonlinear subelliptic problems. For a study of the best constants of such inequalities, the readers can consult \cite{GKR23} and references therein. We recall the following compact embedding result from our previous paper \cite{GKR22}.

\begin{theorem}\label{l-3} Let $\Omega\subset\mathbb{G}$  be an open subset of a stratified Lie group $\G$ of homogeneous dimension $Q$. Then for $0<s<1\leq p<\infty$ such that $\frac{Q}{s}>p$, the embedding  $X_0^{s, p}(\Omega)\hookrightarrow L^r(\Omega)$ is continuous for all $p\leq r\leq p_s^*:=\frac{Qp}{Q-sp}$, that is, there exists  a constant $C=C(Q,s,p, \Omega)>0$ such that for all $u\in X_0^{s, p}(\Omega)$, we have
\begin{equation} \label{emG}
   \|u\|_{L^r(\Omega)}\leq C \|u\|_{X_0^{s,p}(\Omega)}.
\end{equation}
Moreover, if $\Omega$ is bounded, then the embedding is compact for all $r\in[1,p_s^*)$.
\end{theorem} 
\subsection{Example: The Heisenberg group} 
The simplest example of a stratified Lie group is the Heisenberg group $\mathbb{H}^N, N \geq 1,$ which is a non-commutative Lie group. The underlying manifold of $\mathbb{H}^N$ is $\mathbb{R}^{2N+1}:=\mathbb{R}^N\times\mathbb{R}^N\times\mathbb{R}$ for $N\in \mathbb{N}$ but the group operation is different than usual commutative group operation on $\mathbb{R}^{2N+1}.$ Indeed,  for $(x, y, t),(x', y', t')\in \mathbb{H}^N$ the group operation $\circ$ on $\mathbb{H}^N$ is given by

\begin{equation*}
	(x, y, t) \circ(x', y', t')=(x+x', y+y', t+t'+2 (\langle x',y\rangle)-\langle x,y'\rangle),
\end{equation*}
where  $\langle\cdot,\cdot\rangle$ represents the inner product on $\mathbb{R}^N$. We define the dilation map $D_\lambda,$ $\lambda>0$ on $\mathbb{H}^N$ by \begin{equation*}
	D_{\lambda}(x,y,t)=(\lambda x, \lambda y, \lambda^2 t).
\end{equation*}  The Heisenberg group $\mathbb{H}^N$ equipped with canonical dilation map $D_\lambda, \lambda>0$ becomes a homogeneous group. Therefore, it is clear that the homogeneous dimension $Q$ of $\mathbb{H}^N$ is  $2N+2$ while the topological dimension of $\mathbb{H}^N$ is $2N+1.$
The left-invariant vector fields
\begin{align}
	X_i&=\frac{\partial}{\partial x_i}+2y_i\frac{\partial}{\partial t};
	Y_i=\frac{\partial}{\partial y_i}-2x_i\frac{\partial}{\partial t}~\text{and}~
	T=\frac{\partial}{\partial t}, ~\text{for}~i=1, 2,..., N,
\end{align}
form a basis for the Lie algebra corresponding to the Heisenberg group $\mathbb{H}^N.$
Since $[X_i,Y_i]=-4T$ for $i=1,2,...,N$ and $$[X_i,X_j]=[Y_i,Y_j]=[X_i,Y_j]=[X_i,T]=[Y_j,T]=0$$ for all $i\neq j,$ the vector fields $\{X_i, Y_i\}_{i=1}^N$ satisfy the H\"ormander rank condition. Consequently, the sub-Laplacian on $\mathbb{H}^N$  given by 
$$\mathfrak{L}_{\mathbb{H}^N}:=\sum_{i=1}^N (X_i^2+Y_i^2)$$ is a hypoelliptic differential operator.

\section{Some technical results: Weak lower semicontinuity properties } \label{sec3}
 Throughout this section, we fix $s \in (0, 1)$ and assume that $\Om \subset \G$ is a bounded open subset. We will show that the functional defined by 
\begin{equation} \label{funmu}
    \mathfrak{L}_{\mu}(u):= \frac{1}{p} \iint_{\G \times \G} \frac{|u(x)-u(y)|^p}{|y^{-1} x|^{Q+sp}} d x d y -\frac{\mu}{p_s^*} \int_{\Omega} |u(x)|^{p_s^*} d x,
\end{equation}
for every $u \in X_0^{s, p}(\Omega),$ is weak lower semicontinuous. For this objective, let us denote
\begin{align}\label{bestcon}
    C_{Q, s,p^*_s, \Om}:= \sup_{u \in X_0^{s,p}(\Om) \backslash \{0\}} \frac{\|u\|_{L^{p_s^*(\Om)}}}{\Bigg( \iint_{\G \times \G} \frac{|u(x)-u(y)|^p}{|y^{-1} x|^{Q+sp}} d x d y \Bigg)^{1/p}}
\end{align}
as the best constant of the Sobolev embedding $ X_0^{s,p}(\Omega) \hookrightarrow L^{p_s^*}(\Omega)$ given by \eqref{emG}.

Having fixed the above notation, we will now state the main result of this subsection regarding the weak lower semicontinuity property of the functional \eqref{funmu}. These technical results will play a vital role in this paper.
\begin{lemma}\label{weaksemiconti}
Let $s \in (0, 1)$ and $\frac{Q}{s}>p\geq 2$, where $Q$ is the homogeneous dimension of the stratified Lie group $\G.$ Denote by $\overline{B_{X_0^{s,p}(\Omega)}(0, r)}$  the closed ball of radius $r$ centered at $0$ in the subelliptic fractional Sobolev space $X_{0}^{s, p}(\Omega),$ that is, 
\begin{equation}
    \overline{B_{X_0^{s,p}(\Omega)}(0, r)}:= \Bigg\{ u \in X_0^{s,p}(\Om): \Bigg( \iint_{\G \times \G} \frac{|u(x)-u(y)|^p}{|y^{-1} x|^{Q+sp}} d x d y \Bigg)^{1/p} \leq r \Bigg\}.
\end{equation}
Then, for every $\mu>0,$ there exists $r_\mu>0$ such that the functional $\mathfrak{L}_{\mu}$ given by \eqref{funmu} is sequentially weakly lower semicontinuous on $\overline{B_{X_0^{s,p}(\Omega)}(0, r_{\mu})}.$
\end{lemma}
\begin{proof}
We begin the proof by fixing $\mu>0.$ To prove the sequentially weak lower semicontinuity of $\mathfrak{L}_{\mu}$ in $\overline{B_{X_0^{s,p}(\Omega)}(0, r_{\mu})},$ we choose an arbitrary sequence $(u_k)\subset \overline{B_{X_0^{s,p}(\Omega)}(0, r)} $ such that $u_k \rightharpoonup v $ weakly for some $v \in \overline{B_{X_0^{s,p}(\Omega)}(0, r)}.$ This implies that \begin{align} \label{eq3.4}
   & \iint_{\G \times \G} \frac{|u_k(x)-u_k(y)|^{p-2}(u_k(x)-u_k(y)) (\phi(x)-\phi(y))}{|y^{-1} x|^{Q+sp}} d x d y \nonumber\\& \rightarrow \iint_{\G \times \G} \frac{|v(x)-v(y)|^{p-2}(v(x)-v(y)) (\phi(x)-\phi(y))}{|y^{-1} x|^{Q+sp}} d x d y,
\end{align}
for every $\phi \in X^{s,p}_0(\Omega)$ as $k \rightarrow \infty.$ Note that since $s\in(0,1)$ and $1<p<Q$, we have $p_s^*=\frac{Qp}{Q-ps}>p$.  Thus, from the compact embedding as $p<p_s^*$, we get
\begin{equation}\label{eq3.11}
    u_k \rightarrow v~\text{in}~L^p(\Omega)
\end{equation}
Therefore, up to a subsequence
\begin{equation}\label{eq3.12}
    u_k(x) \rightarrow v(x)~\text{a.e in}~\Omega 
\end{equation}
and using a similar argument as in Lemma A.1 \cite{WM97}, we may assume that there exists $g\in L^p(\Omega)$ such that 
\begin{equation}\label{eq3.13}
    |u_k(x)|\leq g(x)~\text{a.e in}~\Omega, \forall\,k \in\mathbb{N}.
\end{equation}
Indeed from \eqref{eq3.12}, we may assume that there exists a subsequence $(w_k)$ of $(u_k)$ such that $$\|w_{j+1}-w_j\|_p\leq \frac{1}{2^j}~\forall\,j\in\mathbb{N}.$$
Define, $g(x)=w_1(x)+\sum_{j=1}^{\infty}\|w_{j+1}(x)-w_j(x)\|_p$. Thus $|w_k(x)|\leq g(x)$ and $|v(x)|\leq g(x).$

Testing \eqref{eq3.4} with $\phi=v$, we obtain 
 \begin{align}\label{eq3.9}
     & \iint_{\G \times \G} \frac{|u_k(x)-u_k(y)|^{p-2}(u_k(x)-u_k(y)) (v(x)-v(y))}{|y^{-1} x|^{Q+sp}} d x d y \nonumber\\& \rightarrow \iint_{\G \times \G} \frac{|v(x)-v(y)|^{p}}{|y^{-1} x|^{Q+sp}} d x d y.
 \end{align}
Notice that for $j=1, 2,3,\dots$, testing \eqref{eq3.4} with $\phi=u_j$ and taking $k\rightarrow\infty$, we get 
 \begin{align}\label{eq3.10}
     & \iint_{\G \times \G} \frac{|u_k(x)-u_k(y)|^{p-2}(u_k(x)-u_k(y)) (u_j(x)-u_j(y))}{|y^{-1} x|^{Q+sp}} d x d y \nonumber\\& \rightarrow \iint_{\G \times \G} \frac{|v(x)-v(y)|^{p-2}(v(x)-v(y)) (u_j(x)-u_j(y))}{|y^{-1} x|^{Q+sp}} d x d y.
 \end{align}
Again from \eqref{eq3.11}, \eqref{eq3.12} and the Lebesgue dominated convergence theorem, we get
\begin{align}\label{eq3.14}
     & \iint_{\G \times \G} \frac{|u_k(x)-u_k(y)|^{p}}{|y^{-1} x|^{Q+sp}} d x d y \rightarrow \iint_{\G \times \G} \frac{|v(x)-v(y)|^{p}}{|y^{-1} x|^{Q+sp}} d x d y.
 \end{align}
Therefore, 
 \begin{align}\label{eq3.15}
     & \iint_{\G \times \G} \frac{|v(x)-v(y)|^{p-2}(v(x)-v(y)) (u_j(x)-u_j(y))}{|y^{-1} x|^{Q+sp}} d x d y \nonumber\\& \rightarrow \iint_{\G \times \G} \frac{|v(x)-v(y)|^{p}}{|y^{-1} x|^{Q+sp}} d x d y.
 \end{align}
To prove the desired conclusion it is enough to show that 
\begin{equation}
    \liminf_{k \rightarrow \infty} \mathcal{L}_{\mu}(u_k) \geq \mathcal{L}_{\mu}(v),
\end{equation}
which means that 
\begin{align}
\nonumber \liminf_{k \rightarrow \infty} \Bigg[ \frac{1}{p} &\Bigg( \iint_{\G \times \G} \frac{|u_k(x)-u_k(y)|^p}{|y^{-1} x|^{Q+sp}} d x d y - \iint_{\G \times \G} \frac{|v(x)-v(y)|^p}{|y^{-1} x|^{Q+sp}} d x d y \Bigg)\\&- \frac{\mu}{p_s^*} \Bigg( \int_{\Omega} |u_k(x)|^{p_s^*} d x-\int_{\Omega} |v(x)|^{p_s^*} d x\Bigg) \Bigg] \geq 0.
\end{align}
First note that  the celebrated Brezis-Lieb lemma (\cite{BL83}) implies that 
 \begin{align} \label{eq3.7}
     \liminf_{k \rightarrow \infty}  \Bigg( \int_{\Omega} |u_k(x)|^{p_s^*} d x-\int_{\Omega} |v(x)|^{p_s^*} d x\Bigg)=\liminf_{k \rightarrow \infty} \int_{\Omega}|u_k(x)-v(x)|^{p_s^*} d x. 
 \end{align}

On the other hand, let us recall the following inequality from \cite[Lemma 4.2]{L90} for $p \geq 2$ and $a, b \in \mathbb{R}:$
\begin{equation}\label{albequ}
    |b|^p-|a|^p \geq p |a|^{p-2} a(b-a)+2^{1-p} |a-b|^p.
\end{equation}
Now, we choose $a:=v(x)-v(y)$ and $b:=u_k(x)-u_k(y)$ for $x, y \in \G.$ Then from \eqref{albequ}
 above we deduce that 
 
 \begin{align}\label{eq3.8}
  \nonumber   \frac{1}{p} &\Bigg( \iint_{\G \times \G} \frac{|u_k(x)-u_k(y)|^p}{|y^{-1} x|^{Q+sp}} d x d y - \iint_{\G \times \G} \frac{|v(x)-v(y)|^p}{|y^{-1} x|^{Q+sp}} d x d y \Bigg) \\&\geq \nonumber\iint_{\G \times \G} \frac{|v(x)-v(y)|^{p-2}(v(x)-v(y)) ((u_k(x)-u_k(y))}{|y^{-1} x|^{Q+sp}} d x d y\\&-\iint_{\G \times \G} \frac{|v(x)-v(y)|^p}{|y^{-1} x|^{Q+sp}} d x d y +\frac{2^{1-p}}{p} \iint_{\G \times \G} \frac{|(u_k-v)(x)-(u_k-v)(y))|^p}{|y^{-1} x|^{Q+sp}} d x d y.
 \end{align}

  Therefore, making use of \eqref{eq3.7} and \eqref{eq3.8} in \eqref{eq3.15} we get 
  \begin{align}\label{eq3.16}
      \liminf_{k \rightarrow \infty}(\mathcal{L}_\mu(u_k)-\mathcal{L}_\mu(v))\geq & \liminf_{k \rightarrow \infty}\left\{ \frac{2^{1-p}}{p} \iint_{\G \times \G} \frac{|(u_k-v)(x)-(u_k-v)(y))|^p}{|y^{-1} x|^{Q+sp}} d x d y\right. \nonumber\\ &\left. -\frac{\mu}{p_s^*}\int_{\Omega}|u_k(x)-v(x)|^{p_s^*} d x\right\}.
  \end{align}
 Now, using the fact $(u_k-v)\subset \overline{B_{X_0^{s,p}(\Omega)}(0, 2r)}$ and the continuous Sobolev embedding $X_0^{s,p}(\Omega)\rightarrow L^{p_s^*}(\Omega)$, from \eqref{eq3.16} we deduce
 \begin{align*}
   \nonumber   \liminf_{k \rightarrow \infty}(\mathcal{L}_\mu(u_k)-\mathcal{L}_\mu(v))\geq & \liminf_{k \rightarrow \infty} \left\{\frac{2^{1-p}}{p}  \|u_k-v\|^p_{X_0^{s,p}(\Omega)} - \frac{\mu}{p_s^*} \|u_k-v\|_{L^{p_s^*}(\Omega)}^{p_s^*}  \right\} 
 \\\geq & \liminf_{k \rightarrow \infty}   \|u_k-v\|^p_{X_0^{s,p}(\Omega)}\left\{ \frac{2^{1-p}}{p} -\frac{\mu c^{p_s^*}}{p_s^*}\|u_k-v\|_{X_0^{s,p}(\Omega)}^{p_s^*-p}\right\}\\
      \geq & \liminf_{k \rightarrow \infty}\|u_k-v\|^p_{X_0^{s,p}(\Omega)}\left\{ \frac{2^{1-p}}{p} -\frac{\mu c^{p_s^*}2^{p_s^*-p}}{p_s^*}r^{p_s^*-p}\right\},
  \end{align*}
  where $c$ is the positive embedding constant.
  
 Choose $$0<r\leq \frac{1}{2}\left( \frac{2^{1-p}p_s^*}{p\mu c^{p_s^*}} \right)^{\frac{1}{p_s^*-p}}.$$
 Then we get $$\liminf_{k \rightarrow \infty}(\mathcal{L}_\mu(u_k)-\mathcal{L}_\mu(v))\geq0.$$
 That is
 \begin{align*}
      \liminf_{k \rightarrow \infty}&\left\{\frac{1}{p}\left(\iint_{\G \times \G} \frac{|u_k(x)-u_k(y)|^{p}}{|y^{-1} x|^{Q+sp}} d x d y - \iint_{\G \times \G} \frac{|v(x)-v(y)|^{p}}{|y^{-1} x|^{Q+sp}} d x d y\right)\right. \\&-\frac{\mu}{p_s^*} \Bigg( \int_{\Omega} |u_k(x)|^{p_s^*} d x-\int_{\Omega} |v(x)|^{p_s^*} d x\Bigg)\geq 0.
 \end{align*}
 Hence, for any $0<r_{\mu}\leq \frac{1}{2}\left( \frac{2^{1-p}p_s^*}{p\mu c^{p_s^*}} \right)^{\frac{1}{p_s^*-p}}$, the functional $\mathcal{L}_\mu$ becomes sequentially weak lower semicontinuous on $\overline{B_{X_0^{s,p}(\Omega)}(0, r_\mu)}.$ This completes the proof.
  \end{proof}

For fixed $\mu, \lambda>0$ we define the functionals
\begin{equation}
    \mathfrak{E}_{p, s, Q}(u):= \left(\iint_{\G \times \G} \frac{|u(x)-u(y)|^p}{|y^{-1}x|^{Q+ps}} d x d y \right)^{\frac{1}{p}}
\end{equation}
and 
\begin{equation} \label{ruz19}
    \mathfrak{F}_{\mu, \lambda}(u):=\frac{\mu}{p_s^*} \int_\Omega  |u(x)|^{p_s^*} d x+\lambda \int_\Omega H(x, u(x)) d x,
\end{equation}
for every $u \in X_0^{s,p}(\Omega),$ where the function $H$ is given by 
\begin{equation} \label{ruz20}
    H(x, t):=\int_0^t h(x , \tau) d \tau, \quad \forall (x, t) \in \Omega \times \mathbb{R}.
\end{equation}

\begin{lemma} \label{Le3.3} Let $p \in (1, \infty),$ $\lambda, \mu>0$ and $0<\zeta<\xi.$ We set
$$\Psi_{\mu, \lambda}(\xi, \zeta):= \sup_{v \in \mathfrak{E}_{p, s, Q}^{-1}([0, \xi] )} \mathfrak{F}_{\mu, \lambda}(v)- \sup_{v \in \mathfrak{E}_{p, s, Q}^{-1}([0, \xi-\zeta]) } \mathfrak{F}_{\mu, \lambda}(v).$$ Suppose that 
\begin{equation} \label{L1}
    \limsup_{\epsilon \rightarrow 0^+}  \frac{\Psi_{\mu, \lambda}(r_0, \epsilon)}{\epsilon}< r_0^{p-1}
\end{equation}
for some $r_0>0.$
    Then 
    \begin{equation}\label{L4}
        \inf_{0<\eta <r_0} \frac{\Psi_{\mu, \lambda}(r_0, r_0-\eta)}{r_0^p-\eta^p}< \frac{1}{p}.
    \end{equation}
\end{lemma}
\begin{proof}
    Let us observe that if $\epsilon \in (0, r_0)$ then we have 
    \begin{align} \label{L2}
        \frac{\Psi_{\mu, \lambda}(r_0, \epsilon)}{r_0^p-(r_0- \epsilon)^p}=  \frac{\Psi_{\mu, \lambda}(r_0, \epsilon)}{\epsilon}  \frac{-\epsilon}{r_0^{p} \left[\left(1-\frac{\epsilon}{r_0} \right)^p-1 \right]}.
    \end{align}
    A simple use of the L'H\^opital rule yields that 
    $$\lim_{\epsilon \rightarrow 0^+} \frac{-\epsilon}{r_0^{p} \left[\left(1-\frac{\epsilon}{r_0} \right)^p-1 \right]} = \frac{1}{p\,r_0^{p-1}}.$$
Therefore, from \eqref{L1} and \eqref{L2} it follows that
\begin{equation}
    \label{L3}
    \limsup_{\epsilon \rightarrow 0^+}\frac{\Psi_{\mu, \lambda}(r_0, \epsilon)}{r_0^p-(r_0- \epsilon)^p}=\limsup_{\epsilon \rightarrow 0^+}\frac{\Psi_{\mu, \lambda}(r_0, \epsilon)}{\epsilon} \times \frac{1}{p\,r_0^{p-1}}< r_0^{p-1} \times \frac{1}{p\,r_0^{p-1}}  = \frac{1}{p}.
\end{equation}
Therefore, we can find  $\bar{\epsilon} \in (0, r_0)$ such that 
$$\frac{\Psi_{\mu, \lambda}(r_0, \bar{\epsilon})}{r_0^p-(r_0- \bar{\epsilon})^p}< \frac{1}{p}$$
and so $\bar{\eta}:=r_0-\bar{\epsilon}<r_0$ gives that $$ \inf_{0<\eta <r_0} \frac{\Psi_{\mu, \lambda}(r_0, r_0-\eta)}{r_0^p-\eta^p}<  \frac{\Psi_{\mu, \lambda}(r_0, r_0-\bar{\eta})}{r_0^p-\bar{\eta}^p}< \frac{1}{p},$$
verifying the inequality \eqref{L4}.
\end{proof}

\begin{lemma} \label{le3.4}
    Let $p \in (1, \infty)$ and $\lambda, \mu >0.$ Suppose that \eqref{L4} holds for some $r_0>0.$ Then we have 
    \begin{equation}
        \label{P1}
        \inf_{u \in \mathfrak{E}^{-1}_{p,s,Q}([0, r_0))} \frac{\sup_{v \in \mathfrak{E}^{-1}_{p,s,Q}([0, r_0])} \mathfrak{F}_{\mu, \lambda}(v)- \mathfrak{F}_{\mu, \lambda}(u)}{r_0^p- \mathfrak{E}_{p,s,Q}(u)^p} <\frac{1}{p}.
    \end{equation}
\end{lemma}
\begin{proof}
    In view  of \eqref{L4}, we have 
    \begin{equation} \label{eq3.26}
        \sup_{v \in \mathfrak{E}^{-1}_{p,s,Q}([0, \eta_0))} \mathfrak{F}_{\mu, \lambda}(v) > \sup_{v \in \mathfrak{E}^{-1}_{p,s,Q}([0, r_0))} \mathfrak{F}_{\mu, \lambda}(v)-\frac{r_0^p-\eta_0^p}{p}
    \end{equation}
    for some $0<\eta_0<r_0.$

    Now, let us recall a basic property of lower semicontinuous functions defined on a topological space $(A, \tau_A).$ Denote the topological closure of a subset $B \subset A$ by $\bar{B}.$ Then, for a real-valued lower semicontinuous function $g$ defined on $\bar{B},$ one have 
    $$\sup_{x \in \bar{B}} g(x) =\sup_{x \in B} g(x).$$ 
    We take $B:=\{ v \in X^{s,p}_0(\Omega): \mathfrak{E}_{p,s,Q}(v) = \eta_0 \},$ so that  the  weak closure of $B$ is $\bar{B}:=\mathfrak{E}^{-1}_{p,s,Q}([0, \eta_0])$ and  $\mathfrak{I}_{\mu, \lambda}$ is sequential weakly lower semicontinuous on $\mathfrak{E}^{-1}_{p,s,Q}([0, \eta_0]).$ Thus, we get 
     
\begin{equation}
        \sup_{v \in \mathfrak{E}^{-1}_{p,s,Q}([0, \eta_0])} \mathfrak{F}_{\mu, \lambda}(v) = \sup_{\{v \in X^{s,p}_0(\Omega): \mathfrak{E}_{p,s,Q}(v)=\eta_0\}} \mathfrak{F}_{\mu, \lambda}(v)=\mathfrak{F}_{\mu, \lambda}(\tilde{v}),
    \end{equation}
    for some $\tilde{v} \in X_0^{s,p}(\Omega)$ with $\mathfrak{E}_{p,s,Q}(\tilde{v})=\eta_0.$
    Thus, from \eqref{eq3.26}, we obtain
    \begin{equation}\label{P3}
        \mathfrak{F}_{\mu, \lambda}(\tilde{v}) > \sup_{v \in \mathfrak{E}^{-1}_{p,s,Q}([0, r_0])} \mathfrak{F}_{\mu, \lambda}(v)-\frac{r_0^p-\mathfrak{E}_{p,s,Q}(\tilde{v})^p}{p}.
    \end{equation}
    Therefore, we get from \eqref{P3} that
    \begin{equation}
        \inf_{u \in \mathfrak{E}^{-1}_{p,s,Q}([0, r_0))} \frac{\sup_{v \in \mathfrak{E}^{-1}_{p,s,Q}([0, r_0])} \mathfrak{F}_{\mu, \lambda}(v)- \mathfrak{F}_{\mu, \lambda}(u)}{r_0^p- \mathfrak{E}_{p,s,Q}(u)^p} < \frac{\sup_{v \in \mathfrak{E}^{-1}_{p,s,Q}([0, r_0])} \mathfrak{F}_{\mu, \lambda}(v)- \mathfrak{F}_{\mu, \lambda}(\tilde{v})}{r_0^p- \mathfrak{E}_{p,s,Q}(\tilde{v})^p} <\frac{1}{p},
    \end{equation}
    proving inequality \eqref{P1}.
\end{proof}

\begin{lemma} \label{Le3.5}
    Let $p \in (1, \infty)$ and $\lambda, \mu>0.$ Suppose that \eqref{P1} holds for some $r_0>0.$ Then, there exists $w \in \mathfrak{E}^{-1}_{p,s,Q}([0, r_0))$ such that 
    \begin{equation} \label{P323}
        \mathfrak{I}_{\lambda, \mu}(w)=\frac{1}{p} \mathfrak{E}_{p,s,Q}(w)^p- \mathfrak{F}_{\mu, \lambda}(w)<\frac{r_0^p}{p}- \mathfrak{F}_{\mu, \lambda}(u)
    \end{equation}
    for every $u \in \mathfrak{E}^{-1}_{p,s,Q}([0, r_0]).$
\end{lemma}

\begin{proof}
    It is clear from the assumption that
    \begin{equation}
        \inf_{u \in \mathfrak{E}^{-1}_{p,s,Q}([0, r_0))} \frac{\sup_{v \in \mathfrak{E}^{-1}_{p,s,Q}([0, r_0])} \mathfrak{F}_{\mu, \lambda}(v)- \mathfrak{F}_{\mu, \lambda}(u)}{r_0^p- \mathfrak{E}_{p,s,Q}(u)^p} <\frac{1}{p}
    \end{equation}
    holds for some $r_0>0.$ This shows that there exists $w \in \mathfrak{E}^{-1}_{p,s,Q}([0, r_0))$ such that 
    \begin{equation}
        \mathfrak{F}_{\mu, \lambda}(u) \leq \sup_{v \in \mathfrak{E}^{-1}_{p,s,Q}([0, r_0])} \mathfrak{F}_{\mu, \lambda}(v) < \mathfrak{F}_{\mu, \lambda}(w)+\frac{r_0^p-\mathfrak{E}_{p,s,Q}(w)^p}{p} 
    \end{equation}
    for every $u \in \mathfrak{E}^{-1}_{p,s,Q}([0, r_0]).$ Therefore, we have
    $$\mathfrak{I}_{\lambda, \mu}(w)<\frac{r_0^p}{p}-\mathfrak{F}_{\mu, \lambda}(u),$$
     for every $u \in \mathfrak{E}^{-1}_{p,s,Q}([0, r_0]).$ 
\end{proof}

\section{ Proof of Theorem \ref{thm1.1}} \label{sec4}

For  a fixed $\mu>0,$ let us define the following  function $g_\mu:(0, \infty) \rightarrow \mathbb{R}$ given by
\begin{equation} \label{rationfun}
    g_{\mu}(r):=\frac{r^{p-1}-\mu C_{Q, s,p^*_s, \Om}^{p_s^*} r^{p_s^*-1} }{a_1 C_{Q, s,p^*_s, \Om} |\Omega|^{\frac{p_s^*-1}{p_s^*}}+a_2C_{Q, s,p^*_s, \Om}^q|\Omega|^{\frac{p_s^*-q}{p_s^*}} r^{q-1}},
\end{equation}
where $C_{Q, s,p^*_s, \Om}$ is the best constant in the fractional subelliptic Sobolev inequality given by \eqref{bestcon}.
It is easy to note that $\lim_{r \rightarrow 0} g_\mu(r)=0$ as $a_1>0$ and $p_s^*>p \geq 2.$ On the other hand, $\lim_{r \rightarrow \infty} g_\mu(r)=-\infty$ from L'H\^opital's rule and $p_s^*>q$. It is easy to see that $g_\mu>0$ in the neighbourhood of $0.$ Indeed, we can find a small enough $r'>0$ (depending on $p$ and $p_s^*$) such that $g_\mu(r)>0$ for $r \in (0, r').$  Therefore, $g_\mu$ has a global maximum, that is,  there exists a  $r_{\mu, \max}$ such that 
$$g_\mu(r_{\mu, \max})=\max_{r>0} g_\mu(r).$$

Set $$r_{0, \mu}:=\min\{r_{\mu, \max}, r_\mu\},$$ where $r_\mu$ is defined by Lemma \ref{weaksemiconti}. By taking any $\lambda < \Lambda_\mu:= g_\mu(r_{0, \mu})$ we can find $r_{0, \mu, \lambda} \in (0, r_{0, \mu})$ such that 
\begin{equation} \label{S42}
    \lambda< \frac{r_{0, \mu, \lambda}^{p-1}-\mu C_{Q, s,p^*_s, \Om}^{p_s^*} r_{0, \mu, \lambda}^{p_s^*-1} }{a_1 C_{Q, s,p^*_s, \Om} |\Omega|^{\frac{p_s^*-1}{p_s^*}}+a_2C_{Q, s,p^*_s, \Om}^q|\Omega|^{\frac{p_s^*-q}{p_s^*}} r_{0, \mu, \lambda}^{q-1}}.
\end{equation}
Now, we take $0<\epsilon< r_{0, \mu, \lambda}$ and define
\begin{equation}
    \Lambda_{\lambda, \mu}(\epsilon, r_{0,\mu, \lambda}):=\frac{\sup_{v \in \mathfrak{E}_{p, s, Q}^{-1}([0, r_{0,\mu, \lambda}] )} \mathfrak{F}_{\mu, \lambda}(v)- \sup_{v \in \mathfrak{E}_{p, s, Q}^{-1}([0, r_{0,\mu, \lambda}-\epsilon]) } \mathfrak{F}_{\mu, \lambda}(v)}{\epsilon}.
\end{equation}
Therefore, by rescaling $v$ and using \eqref{ruz19} and \eqref{ruz20}, we obtain that 
\begin{align*}
    \Lambda_{\lambda, \mu}(\epsilon, r_{0,\mu, \lambda})& \leq \frac{1}{\epsilon} \Bigg| \sup_{v \in \mathfrak{E}_{p, s, Q}^{-1}([0, r_{0,\mu, \lambda}] )} \mathfrak{F}_{\mu, \lambda}(v)- \sup_{v \in \mathfrak{E}_{p, s, Q}^{-1}([0, r_{0,\mu, \lambda}-\epsilon]) } \mathfrak{F}_{\mu, \lambda}(v) \Bigg| \\&
    \leq \sup_{v \in \mathfrak{E}_{p, s, Q}^{-1}([0, 1] )} \int_\Omega \Bigg| \int^{r_{0, \mu, \lambda}v(x)}_{(r_{0,\mu, \lambda}-\epsilon)v(x)} \frac{f_{\mu, \lambda}(x, t)}{\epsilon} d t \Bigg| d x,
\end{align*}
where $f_{\lambda, \mu}(x, t):=\mu |t|^{p_s^*-1}+\lambda h(x, t)$ for a.e. $x \in \Omega$ and for all $t \in \mathbb{R}.$ 
Next, using the subcritical growth condition \eqref{growth} of $h$ along with the subelliptic continuous embedding $X^{s,p}(\Omega) \hookrightarrow L^{p_s^*}(\Omega)$ and H\"older inequality we get
\begin{align*}
    \Lambda_{\lambda, \mu}(\epsilon, r_{0,\mu, \lambda})&\leq \sup_{v \in \mathfrak{E}_{p, s, Q}^{-1}([0, 1] )} \int_\Omega \Big| \int^{r_{0, \mu, \lambda}v(x)}_{(r_{0,\mu, \lambda}-\epsilon)v(x)} \frac{f_{\mu, \lambda}(x, t)}{\epsilon} d t \Big| d x\\& \leq \frac{1}{\epsilon} \sup_{v \in \mathfrak{E}_{p, s, Q}^{-1}([0, 1] )} \int_\Omega \Big| \int^{r_{0, \mu, \lambda}v(x)}_{(r_{0,\mu, \lambda}-\epsilon)v(x)} \left[ \mu |t|^{p_s^*-1} +\lambda (a_1+a_2|t|^{q-1})) \right] d t \Big| d x \\&\leq \frac{\mu}{p_s^*} \sup_{v \in \mathfrak{E}_{p, s, Q}^{-1}([0, 1] )} \|v\|_{L^{p_s^*}(\Omega)}^{p_s^*} \left( \frac{r_{0,\mu, \lambda}^{p_s^*}-(r_{0,\mu, \lambda}-\epsilon)^{p_s^*}}{\epsilon} \right)\\&+ \lambda \sup_{v \in \mathfrak{E}_{p, s, Q}^{-1}([0, 1] )} \Bigg( a_1  \|v\|_{L^{1}(\Omega)} + \frac{a_2}{q}  \|v\|_{L^{q}(\Omega)}^q \left( \frac{r_{0,\mu, \lambda}^{q}-(r_{0,\mu, \lambda}-\epsilon)^{q}}{\epsilon} \right)\Bigg) \\&\leq \frac{\mu C_{Q, s, p_s^*, \Omega}^{p_s^*}}{p_s^*} \left( \frac{r_{0,\mu, \lambda}^{p_s^*}-(r_{0,\mu, \lambda}-\epsilon)^{p_s^*}}{\epsilon} \right)+ \lambda \Bigg( a_1 C_{Q, s, p_s^*, \Omega} |\Omega|^{\frac{p_s^*-1}{p_s^*}}\\&+a_2 \frac{C_{Q, s, p_s^*, \Omega}^q}{q} \left( \frac{r_{0,\mu, \lambda}^{q}-(r_{0,\mu, \lambda}-\epsilon)^{q}}{\epsilon} \right)|\Omega|^{\frac{p_s^*-q}{p_s^*}}\Bigg),
\end{align*} where $C_{Q, s, p_s^*, \Omega}$ is the best constant in the embedding $X^{s,p}(\Omega) \hookrightarrow L^{p_s^*}(\Omega).$ Therefore, 
\begin{align} \label{S44}
 \nonumber   \limsup_{\epsilon \rightarrow 0} \Lambda_{\lambda, \mu}(\epsilon, r_{0,\mu, \lambda}) &\leq \mu C_{Q, s, p_s^*, \Omega}^{p_s^*} r_{0,\mu, \lambda}^{p_s^*-1} + \lambda \Bigg( a_1 C_{Q, s, p_s^*, \Omega} |\Omega|^{\frac{p_s^*-1}{p_s^*}}+a_2 C_{Q, s, p_s^*, \Omega}^q r_{0,\mu, \lambda}^{q-1}|\Omega|^{\frac{p_s^*-q}{p_s^*}}\Bigg)
   \\& < r_{0, \mu, \lambda}^{p-1},
\end{align} by \eqref{S42}. Combining \eqref{S44} with Lemma \ref{Le3.3} and Lemma \ref{le3.4}, we deduce that 
\begin{equation}
     \inf_{u \in \mathfrak{E}^{-1}_{p,s,Q}([0, r_{0, \mu, \lambda}))} \frac{\sup_{v \in \mathfrak{E}^{-1}_{p,s,Q}([0, r_{0, \mu, \lambda}])} \mathfrak{F}_{\mu, \lambda}(v)- \mathfrak{F}_{\mu, \lambda}(u)}{r_{0, \mu, \lambda}^p- \mathfrak{E}_{p,s,Q}(u)^p} <\frac{1}{p}.
\end{equation}
Therefore, from  Lemma \ref{Le3.5} there exists $w_{\mu, \lambda} \in \mathfrak{E}^{-1}_{p,s,Q}([0, r_{0, \mu, \lambda}))$ such that 
    \begin{equation}
        \mathfrak{I}_{\lambda, \mu}(w_{\mu, \lambda})=\frac{1}{p} \mathfrak{E}_{p,s,Q}(w_{\mu, \lambda})^p- \mathfrak{F}_{\mu, \lambda}(w_{\mu, \lambda})<\frac{r_{0, \mu, \lambda}^p}{p}- \mathfrak{F}_{\mu, \lambda}(u)
    \end{equation}
    for every $u \in \mathfrak{E}^{-1}_{p,s,Q}([0, r_{0, \mu, \lambda}]).$ By recalling that  $r_{0, \mu, \lambda}< r_\mu$ and 
$$\mathfrak{I}_{\mu, \lambda}(u)=\mathfrak{L}_{\mu}(u)-\lambda \int_{\Omega} H(x, u(x)) d x,$$ 
we infer from Lemma \ref{weaksemiconti} that the energy functional $\mathfrak{I}_{\mu, \lambda}$ is sequentially weakly lower semicontinuous on $\overline{B_{X_0^{s,p}(\Omega)}(0, r_{0, \mu, \lambda})}.$ Thus, the restriction $\mathfrak{I}_{\mu, \lambda}|_{\overline{B_{X_0^{s,p}(\Omega)}(0, r_{0, \mu, \lambda})}}$ has a global minimum $u_{0, \mu, \lambda}$ in $\overline{B_{X_0^{s,p}(\Omega)}(0, r_{0, \mu, \lambda})}.$ Now, we claim that $u_{0, \mu, \lambda} \in B_{X_0^{s,p}(\Omega)}(0, r_{0, \mu, \lambda}).$ To prove this, suppose that $\mathfrak{E}_{p,s,Q}(u_{0, \mu, \lambda}):=\|u_{0, \mu, \lambda}\|_{X^{s,p}_0(\Omega)}=r_{0, \mu, \lambda}$. Then by \eqref{P323} we have 
$$ \mathfrak{I}_{\lambda, \mu}(u_{0, \mu, \lambda})=\frac{1}{p} r_{0, \mu, \lambda}^p- \mathfrak{F}_{\mu, \lambda}(u_{0, \mu, \lambda})> \mathfrak{I}_{\lambda, \mu}(w_{\mu, \lambda}),$$
contradicting the fact that $u_{0, \mu, \lambda}$ is a global minimum for $\mathfrak{I}_{\mu, \lambda}|_{\overline{B_{X_0^{s,p}(\Omega)}(0, r_{0, \mu, \lambda})}}.$ Therefore, we conclude that $u_{0, \mu,\lambda} \in X^{s,p}_0(\Omega)$ is a local minimum for the energy functional $\mathfrak{I}_{\mu, \lambda}$ with 
$$\mathfrak{E}_{p,s,Q}(u_{0, \mu, \lambda}):=\|u_{0, \mu, \lambda}\|_{X^{s,p}_0(\Omega)}<r_{0, \mu, \lambda}$$ and 
hence, a weak solution to the problem \eqref{pro1}. This completes the proof.\hfill\qed

\section{Some implications and applications} \label{sec5}
This section presents some applications and implications of Theorem \ref{thm1.1}.
The first application of Theorem \ref{thm1.1} we will present is the following result, which determines the sign of the solution in Theorem \ref{thm1.1} provided that we assume an extra condition for the Carath\'eodory function $h$ given by \eqref{growth}.
\begin{theorem} \label{subthm}
     Let $\Omega$ be a bounded open subset of a stratified Lie group $\G$ with the homogeneous dimension $Q.$ Let $s \in (0, 1),$ $\frac{Q}{s} >p \geq 2$ and $q \in [1,p_s^*),$ where $p_s^*=\frac{pQ}{Q-ps}.$ Suppose that $h:\Omega \times \mathbb{R} \rightarrow \mathbb{R}$ is a Carath\'eodory function such that 
    \begin{equation} \label{growth1}
        |h(x, t)| \leq a_1+a_2|t|^{q-1}\quad \text{for}\,\,\text{a.e.}\,\,x\in \Omega,\,\, \forall\, t \in \mathbb{R}
    \end{equation}
    for some $a_1, a_2>0$   and $h(x, 0)=0$ for a.e. $x \in \Omega.$ 
    Then, for every $\mu>0$ and sufficiently small $\lambda>0,$  the following nonlocal problem
    \begin{equation} \label{pro12}
       \begin{cases} 
(-\Delta_{\mathbb{G}, p})^s u= \mu |u|^{p_s^*-2}u+\lambda h(x, u) \quad &\text{in}\quad \Omega, \\
u=0\quad & \text{in}\quad \mathbb{G}\backslash \Omega,
\end{cases}
    \end{equation} has a nonnegative weak solution $u_{\mu, \lambda} \geq 0$ and a nonpositive weak solution $v_{\mu, \lambda} \leq 0$ in $X^{s,p}_0(\Omega).$   In addition, if 
     \begin{equation} \label{growth1-1}
        \liminf\limits_{t\rightarrow 0^+}\frac{H(x,t)}{t^p}=+\infty,~\text{ for a.e. } x\in\Omega, ~\text{where}~ H(x, t):=\int_0^t h(x, \tau)\, d \tau,
    \end{equation} holds, then $u_{\mu, \lambda}$ and $v_{\mu, \lambda}$ are nontrivial.
    
\end{theorem}

\begin{proof}
    We begin by proving the existence of a solution to the problem \eqref{pro12}. Indeed, as we will be dealing with nonnegative solutions, it is sufficient to consider the function
    $$H_+(x, t):=\int_0^t h_+(x, \tau)\, d \tau,$$
    where $$h_+(x, \tau):=\begin{cases}
        h(x, \tau) &\text{if}\quad \tau >0, \\
        0&\text{if} \quad \tau \leq 0,
    \end{cases}$$ for almost every $x \in \Omega$ and $\tau \in \mathbb{R}.$
It is easy to see that $h_+$ is a Carath\'eodory function and satisfies the condition \eqref{growth1} with $h_+(x, 0)=0.$

Now, we define the functional $\mathfrak{I}_{\mu, \lambda}^+:X_0^{s,p}(\Omega) \rightarrow \mathbb{R}$ as 

    \begin{equation}
        \mathfrak{I}_{\mu, \lambda}^+(u):= \frac{1}{p} \iint_{\G \times \G} \frac{|u(x)-u(y)|^p}{|y^{-1} x|^{Q+sp}} d x d y -\frac{\mu}{p_s^*} \int_{\Omega} (u(x)^+)^{p_s^*} d x-\lambda \int_{\Omega} H_+(x, u(x)) d x,
    \end{equation} for every $u \in X^{s,p}_0(\Omega).$
    Note that $\mathfrak{I}_{\mu, \lambda}^+$ is Fr\'echet differentiable and therefore, we have 
     \begin{align} \label{M1}
   \nonumber   \langle  (\mathfrak{I}_{\mu, \lambda}^+)^{'} (u), \varphi \rangle &:= \iint_{\G \times \G} \frac{|u(x)-u(y)|^{p-2}(u(x)-u(y))(\varphi(x)-\varphi(y))}{|y^{-1} x|^{Q+sp}} d x d y \\&-\mu \int_{\Omega} (u(x)^+)^{p_s^*-1}  \varphi(x) d x-\lambda \int_{\Omega} h_+(x, u(x)) \varphi(x) d x,
       \end{align} for every $\varphi \in X^{s,p}_0(\Omega).$ Then, as an application of Theorem \ref{thm1.1}, for every $\mu>0$ and for sufficiently small $\lambda>0,$ there exists a critical point $u_{\mu, \lambda} \in X^{s,p}_0(\Omega)$ of  $\mathfrak{I}_{\mu, \lambda}^+$ and therefore, a weak solution of the problem \eqref{pro12} with $h_+$.

       It remains to show that the obtained solution $u_{\mu, \lambda}$ is nonnegative on $\Omega.$ For this purpose, let us take $\varphi$ as 
       $\varphi:=u_{\mu, \lambda}^{-}:=\max\{-u_{\mu, \lambda}, 0\}.$ We note $u_{\mu, \lambda} \in X^{s,p}_0(\Omega) $ and $||u_{\mu, \lambda}(x)|-|u_{\mu, \lambda}(y)|| \leq |u_{\mu, \lambda}(x)-u_{\mu, \lambda}(y)|$ imply that $|u_{\mu, \lambda}| \in X^{s,p}_0(\Omega).$  Since $|u_{\mu, \lambda}|$ and $u_{\mu, \lambda}$ belong to $X^{s,p}_0(\Omega),$ it is clear, using the fact that $X^{s,p}_0(\Omega)$ is a vector space,  that $u_{\mu, \lambda}^{-} \in X^{s,p}_0(\Omega).$ 
Thus, we test \eqref{M1} with $\varphi= u_{\mu, \lambda}^{-}$  and use the fact that $u_{\mu, \lambda}$ is a critical point of $\mathfrak{I}_{\mu, \lambda}^+,$ that is, $\langle  (\mathfrak{I}_{\mu, \lambda}^+)^{'} (u), \varphi \rangle=0$ for all $\varphi \in X_0^{s,p}(\Omega)$, to  obtain that 
       \begin{align}  \label{ruz5.5}
   \nonumber  0=& \langle  (\mathfrak{I}_{\mu, \lambda}^+)^{'} (u_{\mu, \lambda}), u_{\mu, \lambda}^{-} \rangle \\&= \nonumber  \iint_{\G \times \G} \frac{|u_{\mu, \lambda}(x)-u_{\mu, \lambda}(y)|^{p-2}(u_{\mu, \lambda}(x)-u_{\mu, \lambda}(y))(u_{\mu, \lambda}^{-}(x)-u_{\mu, \lambda}^{-}(y))}{|y^{-1} x|^{Q+sp}} d x d y \nonumber \\&\quad\quad-\mu \int_{\Omega} (u_{\mu, \lambda}(x)^+)^{p_s^*-1}  (u_{\mu, \lambda})(x)^{-} d x-\lambda \int_{\Omega} h_+(x, u_{\mu, \lambda}(x)) u_{\mu, \lambda}^{-}(x) d x\nonumber  \\&\leq  \iint_{\G \times \G} \frac{|u_{\mu, \lambda}(x)-u_{\mu, \lambda}(y)|^{p-2}(u_{\mu, \lambda}(x)-u_{\mu, \lambda}(y))(u_{\mu, \lambda}^{-}(x)-u_{\mu, \lambda}^{-}(y))}{|y^{-1} x|^{Q+sp}} d x d y,
       \end{align}
      using the definition of $h_+$ with the fact the $u_{\mu, \lambda}^{-}=0$ when $u_{\mu, \lambda} \geq 0.$
      Now, we use the inequality $(a-b)(a^--b^-) \leq -(a^- -b^-)^2$ in \eqref{ruz5.5} to get 
      \begin{align*}
          \nonumber  0=& \langle  (\mathfrak{I}_{\mu, \lambda}^+)^{'} (u_{\mu, \lambda}), u_{\mu, \lambda}^{-} \rangle   \\&\leq - \iint_{\G \times \G} \frac{|u_{\mu, \lambda}(x)-u_{\mu, \lambda}(y)|^{p-2}(u_{\mu, \lambda}^{-}(x)-u_{\mu, \lambda}^{-}(y))^2}{|y^{-1} x|^{Q+sp}} d x d y \nonumber \\&\leq
          - \|u_{\mu, \lambda}^{-}\|_{X_0^{s,p}(\Omega)}^p,
      \end{align*}
      
\noindent which implies that 
\begin{align}
     \|u_{\mu, \lambda}^{-}\|_{X_0^{s,p}(\Omega)} = 0,
\end{align} which further shows that $u_{\mu, \lambda}^{-}(x)=0$ for a.e. $x \in \Omega$ leading to the conclusion that $u_{\mu, \lambda}(x)\geq 0$ for a.e. $x \in \Omega.$  
We will now prove that $u_{\mu, \lambda} \not\equiv 0$ in $\Omega$. Consider, $w\in C_c^\infty(\Omega)$ such that $w\geq 0$ and $w\not\equiv 0$. Since, $\liminf\limits_{t\rightarrow 0^+}\frac{H(x,t)}{t^p}=+\infty$ from \eqref{growth1-1}, then for all $M>0$ there exists a $\delta>0$ such that 
$$H(x,t)\geq M t^p,~\forall\, t\in(0,\delta).$$
Thus choosing $t\in(0,\frac{\delta}{\sup_\Omega w})$, we get
\begin{align*}
      \mathfrak{I}_{\mu, \lambda}^+ (tw)&=\frac{1}{p}\|tw\|^p-\frac{\mu}{p_s^*}\int_\Omega|tw|^{p_s^*}dx-\lambda \int_\Omega H^+(x,tw)dx\\
    &\leq \frac{t^p}{p}\|w\|^p-\lambda M \|w\|^p_{L^p(\Omega)}t^p-\frac{\mu}{p_s^*}t^{p_s^*}\int_\Omega|w|^{p_s^*}dx\\
    &<0~\text{for sufficiently large } M>0.
\end{align*}
Therefore, $0$ cannot be a local minimum of $\mathfrak{I}_{\mu, \lambda}^+$ and hence $u_{\mu, \lambda}(x)$ is nontrivial, nonnegative weak solution to the problem \eqref{pro12}.

For the existence of a nonpositive weak solution, we work with the functional $\mathfrak{I}_{\mu, \lambda}^-:X_0^{s,p}(\Omega) \rightarrow \mathbb{R}$ as 

    \begin{equation}
        \mathfrak{I}_{\mu, \lambda}^-(u):= \frac{1}{p} \iint_{\G \times \G} \frac{|u(x)-u(y)|^p}{|y^{-1} x|^{Q+sp}} d x d y -\frac{\mu}{p_s^*} \int_{\Omega} (u(x)^-)^{p_s^*} d x-\lambda \int_{\Omega} H_-(x, u(x)) d x,
    \end{equation} for every $u \in X^{s,p}_0(\Omega).$ Here the function $H_-$ is given by
    $$H_-(x, t):=\int_0^t h_-(x, \tau)\, d \tau,$$
    where $$h_-(x, \tau):=\begin{cases}
        h(x, \tau) &\text{if}\quad \tau <0 \\
        0&\text{if} \quad \tau \geq 0,
    \end{cases}$$ for almost everywhere $x \in \Omega$ and $\tau \in \mathbb{R}.$ By proceeding the same as for the nonnegative solutions, again using Theorem \ref{thm1.1}, we conclude that for every $\mu>0$ and for sufficiently small $\lambda$, there exists a nontrivial, nonpositive weak solution of problem \eqref{pro12}. This concludes the proof of the theorem.
      \end{proof}
    
\begin{remark}
   It is noteworthy to mention that the condition \eqref{growth1-1} is required for obtaining a nontrivial solution to \eqref{pro12}. Thus, any equivalent condition to \eqref{growth1-1} which also satisfies \eqref{growth1} can be assumed for the existence of a non-trivial solution.
\end{remark} 
      Before proceeding to the proof of Theorem \ref{thmnonmain} we note that the proof of Theorem \ref{thm1.1} still holds if we  consider the nonlocal problem involving the lower order term $h: \Omega \times \mathbb{R} \rightarrow \mathbb{R}$ such that 
      \begin{equation}
          |h(x, t)| \leq a_1+a_2 |t|^{m-1}+a_3 |t|^{q-1}\quad \text{a. e.}\,\,\,x \in \Omega,\, \forall t \in \mathbb{R},
      \end{equation}
      for some positive constants $a_j, j=1,2,3$ and $1<m<p \leq q <p_s^*.$ In fact, we have the following results: 
      \begin{theorem} \label{thm1.1dupl}
    Let $\Omega$ be a bounded open subset of a stratified Lie group $G$ with the homogeneous dimension $Q.$ Let $s \in (0, 1),$ $\frac{Q}{s}>p \geq 2.$  Suppose that $m$ and $q$ are two real constants such that 
$$1 \leq  m<p \leq q<p_s^*.$$ Let $h:\Omega \times \mathbb{R} \rightarrow \mathbb{R}$ be a Carath\'eodory function such that 
    \begin{equation} \label{growthdupl}
       |h(x, t)| \leq a_1+a_2 |t|^{m-1}+a_3 |t|^{q-1}\quad \text{a.e. }\,x \in \Omega,\, \forall\, t \in \mathbb{R},
    \end{equation}
    for some $a_1, a_2, a_3>0.$ Then, for every $\mu>0,$ there exists   $\Sigma_\mu>0$ such that the following nonlocal problem
    \begin{equation} \label{pro1dump}
       \begin{cases} 
(-\Delta_{\mathbb{G}, p})^s u= \mu |u|^{p_s^*-2}u+\lambda h(x, u) \quad &\text{in}\quad \Omega, \\
u=0\quad & \text{in}\quad \mathbb{G}\backslash \Omega,
\end{cases}
    \end{equation} admits at least one  weak solution on $u_\lambda \in X^{s,p}_0(\Omega)$ for every $0<\lambda < \Sigma_\mu,$ which is local minimum of the energy functional 
    \begin{equation}
        \mathfrak{I}_{\mu, \lambda}(u):= \frac{1}{p} \iint_{\G \times \G} \frac{|u(x)-u(y)|^p}{|y^{-1} x|^{Q+sp}} d x d y -\frac{\mu}{p_s^*} \int_{\Omega} |u(x)|^{p_s^*} d x-\lambda \int_{\Omega} \int_0^{u(x)} h(x, \tau)d \tau,
    \end{equation} for every $u \in X^{s,p}_0(\Omega).$
\end{theorem}

\begin{proof} The proof of this theorem is a readaptation of the proof of Theorem \ref{thm1.1}. For the completeness of the paper, we present it here. 
    For  a fixed $\mu>0,$ let us define the following  function $g_\mu:(0, \infty) \rightarrow \mathbb{R}$ given by
\begin{equation} \label{rationfun2}
    g_{\mu}(r):=\frac{r^{p-1}-\mu C_{Q, s,p^*_s, \Om}^{p_s^*} r^{p_s^*-1} }{a_1 C_{Q, s,p^*_s, \Om} |\Omega|^{\frac{p_s^*-1}{p_s^*}}+a_2C_{Q, s,p^*_s, \Om}^m|\Omega|^{\frac{p_s^*-m}{p_s^*}} r^{m-1}+a_3C_{Q, s,p^*_s, \Om}^q|\Omega|^{\frac{p_s^*-q}{p_s^*}} r^{q-1}},
\end{equation}
where $C_{Q, s,p^*_s, \Om}$ is the best constant in the fractional subelliptic Sobolev inequality given by \eqref{bestcon}. Since $p_s^*>q \geq p>m \geq 1,$
we note that $\lim_{r \rightarrow 0} g_\mu(r)=0$  and  $\lim_{r \rightarrow \infty} g_\mu(r)=-\infty.$ Arguing similarly to the proof of Theorem \ref{thm1.1}, we deduce that $g_\mu$ has a global maximum, that is,  there exists a  $r_{\mu, \max}$ such that 
$$g_\mu(r_{\mu, \max})=\max_{r>0} g_\mu(r).$$

Set $$r_{0, \mu}:=\min\{r_{\mu, \max}, r_\mu\},$$ where $r_\mu$ is defined by Lemma \ref{weaksemiconti}. By taking any $\lambda < \Sigma_\mu:= g_\mu(r_{0, \mu})$ we can find $r_{0, \mu, \lambda} \in (0, r_{0, \mu})$ such that 
\begin{equation} \label{S42a}
    \lambda< \frac{r_{0, \mu, \lambda}^{p-1}-\mu C_{Q, s,p^*_s, \Om}^{p_s^*} r_{0, \mu, \lambda}^{p_s^*-1} }{a_1 C_{Q, s,p^*_s, \Om} |\Omega|^{\frac{p_s^*-1}{p_s^*}}+a_2C_{Q, s,p^*_s, \Om}^m|\Omega|^{\frac{p_s^*-m}{p_s^*}} r_{0, \mu, \lambda}^{m-1}+a_3C_{Q, s,p^*_s, \Om}^q|\Omega|^{\frac{p_s^*-q}{p_s^*}} r_{0, \mu, \lambda}^{q-1}}.
\end{equation}
Now, we take $0<\epsilon< r_{0, \mu, \lambda}$ and define
\begin{equation}
    \Lambda_{\lambda, \mu}(\epsilon, r_{0,\mu, \lambda}):=\frac{\sup_{v \in \mathfrak{E}_{p, s, Q}^{-1}([0, r_{0,\mu, \lambda}] )} \mathfrak{F}_{\mu, \lambda}(v)- \sup_{v \in \mathfrak{E}_{p, s, Q}^{-1}([0, r_{0,\mu, \lambda}-\epsilon]) } \mathfrak{F}_{\mu, \lambda}(v)}{\epsilon}.
\end{equation}
Therefore, by rescaling $v$ and using \eqref{ruz19} and \eqref{ruz20}, we obtain that 
\begin{align*}
    \Lambda_{\lambda, \mu}(\epsilon, r_{0,\mu, \lambda})& \leq \frac{1}{\epsilon} \Bigg| \sup_{v \in \mathfrak{E}_{p, s, Q}^{-1}([0, r_{0,\mu, \lambda}] )} \mathfrak{F}_{\mu, \lambda}(v)- \sup_{v \in \mathfrak{E}_{p, s, Q}^{-1}([0, r_{0,\mu, \lambda}-\epsilon]) } \mathfrak{F}_{\mu, \lambda}(v) \Bigg| \\&
    \leq \sup_{v \in \mathfrak{E}_{p, s, Q}^{-1}([0, 1] )} \int_\Omega \Bigg| \int^{r_{0, \mu, \lambda}v(x)}_{(r_{0,\mu, \lambda}-\epsilon)v(x)} \frac{f_{\mu, \lambda}(x, t)}{\epsilon} d t \Bigg| d x,
\end{align*}
where $f_{\lambda, \mu}(x, t):=\mu |t|^{p_s^*-1}+\lambda h(x, t)$ for a.e. $x \in \Omega$ and for all $t \in \mathbb{R}.$ 
Next, using the subcritical growth condition \eqref{growth} of $h$ along with the subelliptic continuous embedding $X^{s,p}(\Omega) \hookrightarrow L^{p_s^*}(\Omega)$ and H\"older inequality we get
\begin{align*}
    \Lambda_{\lambda, \mu}(\epsilon, r_{0,\mu, \lambda})&\leq \sup_{v \in \mathfrak{E}_{p, s, Q}^{-1}([0, 1] )} \int_\Omega \Big| \int^{r_{0, \mu, \lambda}v(x)}_{(r_{0,\mu, \lambda}-\epsilon)v(x)} \frac{f_{\mu, \lambda}(x, t)}{\epsilon} d t \Big| d x\\& \leq \frac{1}{\epsilon} \sup_{v \in \mathfrak{E}_{p, s, Q}^{-1}([0, 1] )} \int_\Omega \Big| \int^{r_{0, \mu, \lambda}v(x)}_{(r_{0,\mu, \lambda}-\epsilon)v(x)} \left[ \mu |t|^{p_s^*-1} +\lambda (a_1+ a_2|t|^{m-1}+a_3|t|^{q-1})) \right] d t \Big| d x \\&\leq \frac{\mu}{p_s^*} \sup_{v \in \mathfrak{E}_{p, s, Q}^{-1}([0, 1] )} \|v\|_{L^{p_s^*}(\Omega)}^{p_s^*} \left( \frac{r_{0,\mu, \lambda}^{p_s^*}-(r_{0,\mu, \lambda}-\epsilon)^{p_s^*}}{\epsilon} \right)+ \lambda \sup_{v \in \mathfrak{E}_{p, s, Q}^{-1}([0, 1] )} \Bigg( a_1  \|v\|_{L^{1}(\Omega)}  \\&+\frac{a_2}{m}  \|v\|_{L^{m}(\Omega)}^m \left( \frac{r_{0,\mu, \lambda}^{m}-(r_{0,\mu, \lambda}-\epsilon)^{m}}{\epsilon} \right)+\frac{a_3}{q}  \|v\|_{L^{q}(\Omega)}^q \left( \frac{r_{0,\mu, \lambda}^{q}-(r_{0,\mu, \lambda}-\epsilon)^{q}}{\epsilon} \right)\Bigg) \\&\leq \frac{\mu C_{Q, s, p_s^*, \Omega}^{p_s^*}}{p_s^*} \left( \frac{r_{0,\mu, \lambda}^{p_s^*}-(r_{0,\mu, \lambda}-\epsilon)^{p_s^*}}{\epsilon} \right)\\&+ \lambda \Bigg( a_1 C_{Q, s, p_s^*, \Omega} |\Omega|^{\frac{p_s^*-1}{p_s^*}}+a_2 \frac{C_{Q, s, p_s^*, \Omega}^m}{m} \left( \frac{r_{0,\mu, \lambda}^{m}-(r_{0,\mu, \lambda}-\epsilon)^{m}}{\epsilon} \right)|\Omega|^{\frac{p_s^*-m}{p_s^*}}\\&\quad \quad\quad\quad+a_3 \frac{C_{Q, s, p_s^*, \Omega}^q}{q} \left( \frac{r_{0,\mu, \lambda}^{q}-(r_{0,\mu, \lambda}-\epsilon)^{q}}{\epsilon} \right)|\Omega|^{\frac{p_s^*-q}{p_s^*}}\Bigg),
\end{align*} where $C_{Q, s, p_s^*, \Omega}$ is the best constant in the embedding $X^{s,p}(\Omega) \hookrightarrow L^{p_s^*}(\Omega).$ Therefore, 
\begin{align} \label{S44a}
 \nonumber   \limsup_{\epsilon \rightarrow 0} \Lambda_{\lambda, \mu}(\epsilon, r_{0,\mu, \lambda}) &\leq \mu C_{Q, s, p_s^*, \Omega}^{p_s^*} r_{0,\mu, \lambda}^{p_s^*-1} + \lambda \Bigg( a_1 C_{Q, s, p_s^*, \Omega} |\Omega|^{\frac{p_s^*-1}{p_s^*}}+a_2 C_{Q, s, p_s^*, \Omega}^m r_{0,\mu, \lambda}^{m-1}|\Omega|^{\frac{p_s^*-m}{p_s^*}}\nonumber\\&\quad\quad\quad\quad+a_3 C_{Q, s, p_s^*, \Omega}^q r_{0,\mu, \lambda}^{q-1}|\Omega|^{\frac{p_s^*-q}{p_s^*}}\Bigg)
    < r_{0, \mu, \lambda}^{p-1},
\end{align} by \eqref{S42a}. Combining \eqref{S44a} with Lemma \ref{Le3.3} and Lemma \ref{le3.4}, we deduce that 
\begin{equation}
     \inf_{u \in \mathfrak{E}^{-1}_{p,s,Q}([0, r_{0, \mu, \lambda}))} \frac{\sup_{v \in \mathfrak{E}^{-1}_{p,s,Q}([0, r_{0, \mu, \lambda}])} \mathfrak{F}_{\mu, \lambda}(v)- \mathfrak{F}_{\mu, \lambda}(u)}{r_{0, \mu, \lambda}^p- \mathfrak{E}_{p,s,Q}(u)^p} <\frac{1}{p}.
\end{equation}
Therefore, from  Lemma \ref{Le3.5} there exists $w_{\mu, \lambda} \in \mathfrak{E}^{-1}_{p,s,Q}([0, r_{0, \mu, \lambda}))$ such that 
    \begin{equation}
        \mathfrak{I}_{\lambda, \mu}(w_{\mu, \lambda})=\frac{1}{p} \mathfrak{E}_{p,s,Q}(w_{\mu, \lambda})^p- \mathfrak{F}_{\mu, \lambda}(w_{\mu, \lambda})<\frac{r_{0, \mu, \lambda}^p}{p}- \mathfrak{F}_{\mu, \lambda}(u)
    \end{equation}
    for every $u \in \mathfrak{E}^{-1}_{p,s,Q}([0, r_{0, \mu, \lambda}]).$ By recalling that  $r_{0, \mu, \lambda}< r_\mu$ and 
$$\mathfrak{I}_{\mu, \lambda}(u)=\mathfrak{L}_{\mu}(u)-\lambda \int_{\Omega} H(x, u(x)) d x,$$ 
we infer from Lemma \ref{weaksemiconti} that the energy functional $\mathfrak{I}_{\mu, \lambda}$ is sequentially weakly lower semicontinuous on $\overline{B_{X_0^{s,p}(\Omega)}(0, r_{0, \mu, \lambda})}.$ Thus, the restriction $\mathfrak{I}_{\mu, \lambda}|_{\overline{B_{X_0^{s,p}(\Omega)}(0, r_{0, \mu, \lambda})}}$ has a global minimum $u_{0, \mu, \lambda}$ in $\overline{B_{X_0^{s,p}(\Omega)}(0, r_{0, \mu, \lambda})}.$ Now, we claim that $u_{0, \mu, \lambda} \in B_{X_0^{s,p}(\Omega)}(0, r_{0, \mu, \lambda}).$ To prove this, suppose that $\mathfrak{E}_{p,s,Q}(u_{0, \mu, \lambda}):=\|u_{0, \mu, \lambda}\|_{X^{s,p}_0(\Omega)}=r_{0, \mu, \lambda}$. Then by \eqref{P323} we have 
$$ \mathfrak{I}_{\lambda, \mu}(u_{0, \mu, \lambda})=\frac{1}{p} r_{0, \mu, \lambda}^p- \mathfrak{F}_{\mu, \lambda}(u_{0, \mu, \lambda})> \mathfrak{I}_{\lambda, \mu}(w_{\mu, \lambda}),$$
contradicting the fact that $u_{0, \mu, \lambda}$ is a global minimum for $\mathfrak{I}_{\mu, \lambda}|_{\overline{B_{X_0^{s,p}(\Omega)}(0, r_{0, \mu, \lambda})}}.$ Therefore, we conclude that $u_{0, \mu,\lambda} \in X^{s,p}_0(\Omega)$ is a local minimum for the energy functional $\mathfrak{I}_{\mu, \lambda}$ with 
$$\mathfrak{E}_{p,s,Q}(u_{0, \mu, \lambda}):=\|u_{0, \mu, \lambda}\|_{X^{s,p}_0(\Omega)}<r_{0, \mu, \lambda}$$ and 
hence, a weak solution to the problem \eqref{pro1}. This completes the proof.
\end{proof}
\begin{theorem} 
    Let $\Omega$ be a bounded open subset of a stratified Lie group $\G$ with the
homogeneous dimension $Q.$ Let  $s \in (0,1)$ and  $\frac{Q}{s}>p\geq 2.$ Suppose that $m$ and $q$ are two real constants such that 
$$1 \leq m<p\leq q<p_s^*.$$
Then, there exists an open interval $\Sigma \subset (0, +\infty)$ such that, for every $\lambda \in \Sigma,$ the subelliptic nonlocal problem 
\begin{equation} \label{pro15.7}
       \begin{cases} 
(-\Delta_{\mathbb{G}, p})^s u=  |u|^{p_s^*-2}u+\lambda (|u|^{m-1}+|u|^{q-1}) \quad &\text{in}\quad \Omega \\
u=0\quad & \text{in}\quad \mathbb{G}\backslash \Omega
\end{cases},
    \end{equation} has at least a nonnegative nontrivial weak solution $u_\lambda \in X_{0}^{s,p}(\Omega),$ which is a  local minimum of the energy functional 
     \begin{equation}
        \mathfrak{I}_{ \lambda}(u):= \frac{1}{p} \iint_{\G \times \G} \frac{|u(x)-u(y)|^p}{|y^{-1} x|^{Q+sp}} d x d y -\frac{1}{p_s^*} \int_{\Omega} |u(x)|^{p_s^*} d x-\frac{\lambda}{m} \int_{\Omega} |u(x)|^m  d x-\frac{\lambda}{q} \int_{\Omega} |u(x)|^q  d x.
    \end{equation} 
\end{theorem}
\begin{proof} We first note that the statement and proof of Theorem \ref{subthm} holds if $h:\Omega \times \mathbb{R} \rightarrow \mathbb{R}$ is a Carath\'eodory function such that 
    \begin{equation} \label{growthdupl1}
       |h(x, t)| \leq a_1+a_2 |t|^{m-1}+a_3 |t|^{q-1}\quad \text{a.e.}\,\,\,x \in \Omega,\, \forall t \in \mathbb{R},
    \end{equation}
    for some $a_1, a_2, a_3>0.$   In our case $h(x, t)=|t|^{m-1}+|t|^{q-1}.$  In this case, $h(x, 0)=0$ for a.e. $x \in \Omega$ and  $u \equiv 0$ is a  trivial solution of \eqref{pro15.7}. The proof of this theorem follows from the fact that if zero is not a local minimum for the energy functional $\mathfrak{I}_{\lambda}$ then any weak solution of problem \eqref{pro15.7} obtained by Theorem \ref{thm1.1dupl} is nontrivial as a consequence of Theorem \ref{subthm}.
   In particular, we observe that zero is not a local minimum   for the energy functional $\mathfrak{I}_{\mu, \lambda}^+$ given by 
 \begin{equation}
        \mathfrak{I}_{\lambda}^+(u):= \frac{1}{p} \iint_{\G \times \G} \frac{|u(x)-u(y)|^p}{|y^{-1} x|^{Q+sp}} d x d y -\frac{1}{p_s^*} \int_{\Omega} (u(x)^+)^{p_s^*} d x-\lambda \int_{\Omega} H_+(x, u(x)) d x,
    \end{equation} for every $u \in X^{s,p}_0(\Omega),$ where $$H_+(x, u):= \begin{cases}
         \frac{u^m}{m}+\frac{u^q}{q}\quad \text{if} \quad u>0, \\
         0 \quad \text{if} \quad u \leq 0.
    \end{cases}$$
    To see that zero is not a local minimum, let us fix a nonnegative function $u_0 \in X^{s, p}_0(\Omega) \backslash \{0\}.$ Then, for sufficiently small $\tau>0,$ we get 
   \begin{align*}
         \mathfrak{I}_{\lambda}^+(\tau u_0)&= \frac{\tau^p}{p} \iint_{\G \times \G} \frac{|u_0(x)-u_0(y)|^p}{|y^{-1} x|^{Q+sp}} d x d y\\& -\frac{\tau^{p_s^*}}{p_s^*} \int_{\Omega} (u_0(x)^+)^{p_s^*} d x- \frac{\lambda \tau^m}{m} \int_{\Omega} u_0(x)^m\, d x-\frac{\lambda \tau^q}{q} \int_{\Omega} u_0(x)^q\, d x<0,
        \end{align*}
     as $1 \leq m<p\leq q<p_s^*.$  This completes the proof of the theorem. \end{proof}

\section{Subcritical nonlocal Brezis-Nirenberg problem }
In this section, we consider the problem \eqref{pro1intro} for $\mu=0,$  that is, we investigate the problem \begin{equation} \label{prob6.1}
      \begin{cases} 
(-\Delta_{\mathbb{G}, p})^s u= \lambda h(x, u) \quad &\text{in}\quad \Omega, \\
u=0\quad & \text{in}\quad \mathbb{G}\backslash \Omega,
\end{cases}
  \end{equation}
where $ \lambda>0$ is a real parameter and $h$ is a subcritical nonlinearity satisfying \begin{equation} \label{growthsec6}
        |h(x, t)| \leq a_1+a_2|t|^{q-1}\quad \text{for}\,\,\text{a.e.}\,\,x\in \Omega,\,\, \forall t \in \mathbb{R}
    \end{equation}
    for some $a_1, a_2>0$ and $q \in [1, p_s^*).$

We consider the energy function $\mathfrak{I}_{\lambda}: X_0^{s, p}(\Omega) \rightarrow \mathbb{R}$ defined by 
\begin{equation}
        \mathfrak{I}_{\lambda}(u):= \frac{1}{p} \iint_{\G \times \G} \frac{|u(x)-u(y)|^p}{|y^{-1} x|^{Q+sp}} d x d y -\lambda \int_{\Omega} \int_0^{u(x)} h(x, \tau)d \tau,
    \end{equation} for every $u \in X^{s,p}_0(\Omega).$ As a consequence of our hypothesis \eqref{growthsec6} for the growth of $h,$ the functional $\mathfrak{I}_\lambda \in C^1(X_0^{s,p}(\Omega))$ and  its derivative at $u \in X_0^{s, p}(\Omega)$ is given by 
    \begin{align}
        \langle \mathfrak{I}_\lambda'(u), \varphi \rangle& \nonumber= \int_{\G \times \G} \frac{|u(x)-u(y)|^{p-2}(u(x)-u(y)) (\varphi(x)-\varphi(y))}{|y^{-1} \circ x|^{Q+sp}} d x d y\\& \quad\quad   - \lambda \int_{G} h(x, u(x)) \varphi(x) d x
    \end{align}
for every $\varphi \in X_{0}^{s, p}(\Omega).$ Recall that, we say that a function $u \in X_{0}^{s, p}(\Omega)$ is a weak solution of \eqref{prob6.1} if 
    \begin{equation}
        \int_{\G \times \G} \frac{|u(x)-u(y)|^{p-2}(u(x)-u(y)) (\varphi(x)-\varphi(y))}{|y^{-1} \circ x|^{Q+sp}} d x d y = \lambda \int_{G} h(x, u(x)) \varphi(x) d x
    \end{equation}
    for all $\varphi \in X_{0}^{s, p}(\Omega).$ In other words, weak solutions are the critical points of the energy functional $\mathfrak{I}_\lambda,$ that is, $u \in X_{0}^{s, p}(\Omega)$ such that  $\langle \mathfrak{I}_\lambda(u), \varphi \rangle=0$ for every $\varphi \in X_{0}^{s, p}(\Omega).$ 

Similar to the previous section, 
for a fixed $\lambda>0$ we define the functional 
\begin{equation}
    \mathfrak{E}_{p, s, Q}(u):= \left(\iint_{\G \times \G} \frac{|u(x)-u(y)|^p}{|y^{-1}x|^{Q+ps}} d x d y \right)^{\frac{1}{p}}
\end{equation}
and 
\begin{equation}
    \mathfrak{F}_{ \lambda}(u):=\lambda \int_\Omega H(x, u(x)) d x,
\end{equation}
for every $u \in X_0^{s,p}(\Omega),$ where the function $H$ is given by 
\begin{equation}
    H(x, t):=\int_0^t h(x , \tau) d \tau, \quad \forall (x, t) \in \Omega \times \mathbb{R}.
\end{equation}

It is worth noting that, due to condition \eqref{growthsec6}, the operator $\mathfrak{F}_{\lambda}$ is well-defined and sequentially weakly (upper) continuous. Consequently, the operator $\mathfrak{I}_\lambda$ is sequentially weakly lower semicontinuous on $X_0^{s,p}(\Omega)$. With the aforementioned notations, we can establish the following two lemmas, which will be pivotal in the subsequent analysis.

\begin{lemma} \label{Le6.6} Let $p \in (1, \infty),$ $\lambda>0$ and $0<\zeta<\xi.$ We set
$$\Psi_{\lambda}(\xi, \zeta ):= \sup_{v \in \mathfrak{E}_{p, s, Q}^{-1}([0, \xi] )} \mathfrak{F}_{\lambda}(v)- \sup_{v \in \mathfrak{E}_{p, s, Q}^{-1}([0, \xi-\zeta]) } \mathfrak{F}_{\lambda}(v).$$ Suppose that \begin{equation} \label{L6.1}
    \limsup_{\epsilon \rightarrow 0^+}  \frac{\Psi_{\lambda}(r_0, \epsilon)}{\epsilon}< r_0^{p-1}
\end{equation}
for some $r_0>0.$
    Then 
    \begin{equation}\label{L6.4}
        \inf_{0<\eta <r_0} \frac{\Psi_{\lambda}(r_0, r_0-\eta)}{r_0^p-\eta^p}< \frac{1}{p}.
    \end{equation}
\end{lemma}

\begin{lemma} \label{le6.7}
    Let $p \in (1, \infty)$ and $\lambda >0.$ Suppose that \eqref{L6.4} holds for some $r_0>0.$ Then we have 
    \begin{equation}
        \inf_{u \in \mathfrak{E}^{-1}_{p,s,Q}([0, r_0))} \frac{\sup_{v \in \mathfrak{E}^{-1}_{p,s,Q}([0, r_0])} \mathfrak{F}_{ \lambda}(v)- \mathfrak{F}_{\lambda}(u)}{r_0^p- \mathfrak{E}_{p,s,Q}(u)^p} <\frac{1}{p}.
    \end{equation}
\end{lemma}

The proofs of these two lemmas are identical with Lemma \ref{Le3.3} and Lemma \ref{le3.4} and therefore, we skip the proof here. Our main result of this section is stated below. 
    \begin{theorem} \label{thm6.5}
    Let $\Omega$ be a bounded open subset of a stratified Lie group $G$ with the homogeneous dimension $Q.$ Let $s \in (0, 1),$ $1<p<q<p_s^*$ and $ps<Q,$ where $p_s^*:=\frac{Qp}{Q-ps}.$ Let $h:\Omega \times \mathbb{R} \rightarrow \mathbb{R}$ be a Carath\'eodory function such that 
    \begin{equation} 
        |h(x, t)| \leq a_1+a_2|t|^{q-1}\quad \text{for}\,\,\text{a.e.}\,\,x\in \Omega,\,\, \forall t \in \mathbb{R}
    \end{equation}
    for some $a_1, a_2>0.$ Furthermore, let 
    $$0<\lambda < \frac{(q-p)^{\frac{q-p}{q-1}} (p-1)^{\frac{p-1}{q-1}}}{(a_1 C1)^{\frac{q-p}{q-1}} (a_2C_2)^{\frac{p-1}{q-1}} (q-1) C_{Q, s, p_s^*, \Omega}^p |\Omega|^{\frac{p_s^*-q}{p_s^*}\left( \frac{p-1}{q-1} \right)} |\Omega|^{\frac{p_s^*-1}{p_s^*} \left( \frac{q-p}{q-1} \right)}  },$$ where $C_1,$ $C_2$ and  $C_{Q, s, p_s^* , \Omega}$ is the embedding constant of $L^{p_s^*}(\Omega) \hookrightarrow L^1(\Omega),$ $L^{p_s^*}(\Omega) \hookrightarrow L^q(\Omega),$ and  $X_{0}^{s, p}(\Omega) \hookrightarrow L^{p_s^*}(\Omega),$ respectively and $|\Omega|$ is the Haar measure of the set $\Omega.$ Then, the following  nonlocal subelliptic parameter  problem
    \begin{equation} 
       \begin{cases} 
(-\Delta_{\mathbb{G}, p})^s u= \lambda h(x, u) \quad &\text{in}\quad \Omega \\
u=0\quad & \text{in}\quad \mathbb{G}\backslash \Omega
\end{cases},
    \end{equation} admits a weak solution  $u_{0, \lambda} \in X^{s,p}_0(\Omega)$ and 
    $$\|u_{0,  \lambda}\|_{X^{s,p}_0(\Omega)}< \left( \frac{\lambda (q-1) a_2 C_{Q,s, r, \Omega}^q |\Omega|^{\frac{r-q}{r}}}{p-1} \right)^{\frac{1}{p-q}}.$$ 
\end{theorem}

\begin{proof} We fix $\lambda$ such that 
\begin{equation} \label{S42sec6}
    0<\lambda< \frac{(q-p)^{\frac{q-p}{q-1}} (p-1)^{\frac{p-1}{q-1}}}{(a_1 C_1)^{\frac{q-p}{q-1}} (a_2C_2)^{\frac{p-1}{q-1}} (q-1) C_{Q, s, p_s^*, \Omega}^p |\Omega|^{\frac{p_s^*-q}{p_s^*}\left( \frac{p-1}{q-1} \right)} |\Omega|^{\frac{p_s^*-1}{p_s^*} \left( \frac{q-p}{q-1} \right)}  }.
\end{equation}
Now, let us consider $0<\epsilon< r$ and define
\begin{equation}
    \Lambda_{\lambda, \mu}(\epsilon, r):=\frac{\sup_{v \in \mathfrak{E}_{p, s, Q}^{-1}([0, r] )} \mathfrak{F}_{ \lambda}(v)- \sup_{v \in \mathfrak{E}_{p, s, Q}^{-1}([0, r-\epsilon]) } \mathfrak{F}_{ \lambda}(v)}{\epsilon}.
\end{equation}
Therefore, by rescaling $v$ we obtain that 
\begin{align*}
    \Lambda_{\lambda}(\epsilon, r)&\leq\frac{1}{\epsilon} \Bigg| \sup_{v \in \mathfrak{E}_{p, s, Q}^{-1}([0, r] )} \mathfrak{F}_{ \lambda}(v)- \sup_{v \in \mathfrak{E}_{p, s, Q}^{-1}([0, r-\epsilon]) } \mathfrak{F}_{ \lambda}(v) \Bigg| \\&
    \leq \sup_{v \in \mathfrak{E}_{p, s, Q}^{-1}([0, 1] )} \int_\Omega \Bigg| \int^{r v(x)}_{(r-\epsilon)v(x)} \frac{\lambda h(x, t)}{\epsilon} d t \Bigg| d x.
\end{align*}

Next, using the subcritical growth condition \eqref{growthsec6} of $h$ along with the subelliptic  embedding $X_0^{s,p}(\Omega) \hookrightarrow L^{p_s^*}(\Omega),$  and H\"older inequality for exponents $\frac{p_s^*}{p_s^*-q}$ and $\frac{p_s^*}{q}$,  we get
\begin{align*}
    \Lambda_{\lambda}(\epsilon, r)&\leq \sup_{v \in \mathfrak{E}_{p, s, Q}^{-1}([0, 1] )} \int_\Omega \Big| \int^{rv(x)}_{(r-\epsilon)v(x)} \frac{\lambda h(x, t)}{\epsilon} d t \Big| d x\\& \leq \frac{1}{\epsilon} \sup_{v \in \mathfrak{E}_{p, s, Q}^{-1}([0, 1] )} \int_\Omega \Big| \int^{rv(x)}_{(r-\epsilon)v(x)}  \lambda (a_1+a_2|t|^{q-1}))  d t \Big| d x \\&\leq  \lambda \sup_{v \in \mathfrak{E}_{p, s, Q}^{-1}([0, 1] )} \Bigg( a_1 \|v\|_{L^1(\Omega)} + \frac{a_2}{q} \left( \frac{r^{q}-(r-\epsilon)^{q}}{\epsilon} \right) \|v\|_{L^q(\Omega)}^q \Bigg) \\&\leq  \lambda \Bigg( a_1 C_1 C_{Q, s, p_s^*, \Omega} |\Omega|^{\frac{p_s^*-1}{p_s^*}}+a_2 C_2 \frac{C_{Q, s, p_s^*, \Omega}^q}{q} \left( \frac{r^{q}-(r-\epsilon)^{q}}{\epsilon} \right)|\Omega|^{\frac{r-q}{r}}\Bigg),
\end{align*} where $C_1, C_2$ and $C_{Q, s, p_s^*, \Omega}$ are the embedding constants.  Therefore, 
\begin{align} \label{S44sec6}
   \limsup_{\epsilon \rightarrow 0} \Lambda_{\lambda}(\epsilon, r) &\leq \lambda \Bigg( a_1 C_1 C_{Q, s, p_s^*, \Omega} |\Omega|^{\frac{p_s^*-1}{p_s^*}}+a_2 C_2C_{Q, s, p_s^*, \Omega}^q r^{q-1}|\Omega|^{\frac{p_s^*-q}{p_s^*}}\Bigg).
\end{align} 
Now consider the real-valued function
$$\varphi_{\lambda}(r)=\lambda \Bigg( a_1 C_1 C_{Q, s, p_s^*, \Omega} |\Omega|^{\frac{p_s^*-1}{p_s^*}}+a_2 C_2 C_{Q, s, p_s^*, \Omega}^q r^{q-1}|\Omega|^{\frac{p_s^*-q}{p_s^*}}\Bigg)-r^{p-1},$$
for every $r>0.$ By using the elementary calculus, one can see that $\inf_{r>0} \varphi_\lambda(r)$ is attained at 
$$r_{0, \lambda}:=\left( \frac{\lambda (q-1) a_2 C_2 C_{Q,s,p_s^*, \Omega}^q |\Omega|^{\frac{p_s^*-q}{p_s^*}}}{p-1} \right)^{\frac{1}{p-q}}.$$
On the other hand,  by \eqref{S42sec6}, we see that 
$$\inf_{r>0} \varphi_\lambda(r)= \varphi_\lambda(r_{0, \lambda})<0.$$
Thus, \eqref{S44sec6} yields that 
\begin{equation} \label{humha}
    \limsup_{\epsilon \rightarrow 0} \Lambda_{\lambda}(\epsilon, r) \leq r_{0, \lambda}^{p-1}.
\end{equation}
Combining \eqref{humha} with Lemma \ref{Le3.3} and Lemma \ref{le3.4}, we deduce that 
\begin{equation} 
     \inf_{u \in \mathfrak{E}^{-1}_{p,s,Q}([0, r_{0, \lambda}))} \frac{\sup_{v \in \mathfrak{E}^{-1}_{p,s,Q}([0, r_{0, \lambda}])} \mathfrak{F}_{ \lambda}(v)- \mathfrak{F}_{\lambda}(u)}{r_{0, \lambda}^p- \mathfrak{E}_{p,s,Q}(u)^p} <\frac{1}{p}.
\end{equation}
Therefore, the above relation shows that there exists $w_\lambda \in \mathfrak{E}^{-1}_{p,s,Q}([0, r_{0, \lambda}))$ such that 
    \begin{equation}
        \mathfrak{F}_{ \lambda}(u) \leq \sup_{v \in \mathfrak{E}^{-1}_{p,s,Q}([0, r_{0, \lambda}])} \mathfrak{F}_{ \lambda}(v) < \mathfrak{F}_{ \lambda}(w_\lambda)+\frac{r_{0, \lambda}^p-\mathfrak{E}_{p,s,Q}(w_\lambda)^p}{p} 
    \end{equation}
    for every $u \in \mathfrak{E}^{-1}_{p,s,Q}([0, r_{0, \lambda}]).$ Thus,  we deduce that
    \begin{equation} \label{eqsec6.17}
        \mathfrak{I}_{\lambda}(w_{\lambda})=\frac{1}{p} \mathfrak{E}_{p,s,Q}(w_{ \lambda})^p- \mathfrak{F}_{ \lambda}(w_{\lambda})<\frac{r_{0,  \lambda}^p}{p}- \mathfrak{F}_{ \lambda}(u)
    \end{equation}
    for every $u \in \mathfrak{E}^{-1}_{p,s,Q}([0, r_{0,  \lambda}]).$ Since the energy functional $\mathfrak{I}_{ \lambda}$ is sequentially weakly lower semicontinuous, its restriction on $\mathfrak{E}^{-1}_{p,s,Q}([0, r_{0,  \lambda}])$  has a global minimum $u_{0, \lambda}$ in $\mathfrak{E}^{-1}_{p,s,Q}([0, r_{0,  \lambda}]).$ Note that $u_{0, \mu, \lambda} \in \mathfrak{E}^{-1}_{p,s,Q}([0, r_{0,  \lambda}]).$ Indeed, suppose that $$\mathfrak{E}_{p,s,Q}(u_{0, \lambda}):=\|u_{0, \lambda}\|_{X^{s,p}_0(\Omega)}=r_{0,  \lambda},$$ 
    then by \eqref{eqsec6.17} we have 
$$ \mathfrak{I}_{\lambda}(u_{0,  \lambda})=\frac{1}{p} r_{0,  \lambda}^p- \mathfrak{F}_{\lambda}(u_{0,  \lambda})> \mathfrak{I}_{\lambda}(w_{ \lambda}),$$
contradicting the fact that $u_{0, \lambda}$ is a global minimum for $\mathfrak{I}_{\lambda}.$ Therefore, we conclude that $u_{0,\lambda} \in X^{s,p}_0(\Omega)$ is a local minimum for the energy functional $\mathfrak{I}_{ \lambda}$ with 
$$\mathfrak{E}_{p,s,Q}(u_{0, \lambda}):=\|u_{0,  \lambda}\|_{X^{s,p}_0(\Omega)}<r_{0, \lambda}$$ and 
hence, a weak solution to the problem \eqref{pro1}. This completes the proof. \end{proof}

The following interesting result is a special case $(\lambda=1)$ of Theorem \ref{thm6.5}. Specifically, this is a nonlocal version of the local problem studied in \cite{AC04} within the Euclidean setting, using the variational principle obtained by Ricceri \cite{Ric00}.

\begin{cor} 
    Let $\Omega$ be a bounded open subset of a stratified Lie group $G$ with the homogeneous dimension $Q.$ Let $s \in (0, 1),$ $1<p<q<p_s^*=\frac{pQ}{Q-ps}$ and $\frac{Q}{s}>p\geq 2.$ Let $h:\Omega \times \mathbb{R} \rightarrow \mathbb{R}$ be a Carath\'eodory function such that 
    \begin{equation} 
        |h(x, t)| \leq a_1+a_2|t|^{q-1}\quad \text{for}\,\,\text{a.e.}\,\,x\in \Omega,\,\, \forall t \in \mathbb{R}
    \end{equation}
    for some $a_1, a_2>0.$  Assume that $$ a_1^{\frac{q-p}{p-1}} a_2 < \frac{(q-p)^{\frac{q-p}{p-1}}(p-1)}{  C_1^{\frac{q-p}{p-1}} C_2 (q-1)^{\frac{q-1}{p-1}} C_{Q, s, r, \Omega}^{p \left(\frac{q-1}{p-1} \right)} |\Omega|^{\frac{p_s^*-q}{p_s^*}} |\Omega|^{\frac{p_s^*-1}{p_s^*} \left( \frac{q-p}{p-1} \right)} },$$ where $C_1,$ $C_2$ and  $C_{Q, s, p_s^* , \Omega}$ is the embedding constant of $L^{p_s^*}(\Omega) \hookrightarrow L^1(\Omega),$ $L^{p_s^*}(\Omega) \hookrightarrow L^q(\Omega),$ and  $X_{0}^{s, p}(\Omega) \hookrightarrow L^{p_s^*}(\Omega),$ respectively  and $|\Omega|$ is the Haar measure of set $\Omega.$ 
     Then, the following  nonlocal subelliptic problem
    \begin{equation} 
       \begin{cases} 
(-\Delta_{\mathbb{G}, p})^s u=  h(x, u) \quad &\text{in}\quad \Omega \\
u=0\quad & \text{in}\quad \mathbb{G}\backslash \Omega
\end{cases},
    \end{equation} admits a weak solution on $u_{0} \in X^{s,p}_0(\Omega)$ and 
    $$\|u_{0}\|_{X^{s,p}_0(\Omega)}< \left( \frac{(q-1) a_2 C_2 C_{Q,s,p_s^*, \Omega}^q |\Omega|^{\frac{p_s^*-q}{p_s^*}}}{p-1} \right)^{\frac{1}{p-q}}.$$

\end{cor}

\section*{Conflict of interest statement}
On behalf of all authors, the corresponding author states that there is no conflict of interest.

\section*{Data availability statement}
Data sharing is not applicable to this article as no datasets were generated or analysed during the current study.

\section*{Acknowledgement}
 SG acknowledges the research facilities available at the Department of Mathematics, NIT Calicut under the DST-FIST support, Govt. of India [Project no. SR/FST/MS-1/2019/40 Dated. 07.01 2020.].
  VK and MR are supported by the FWO Odysseus 1 grant G.0H94.18N: Analysis and Partial Differential Equations, the Methusalem programme of the Ghent University Special Research Fund (BOF) (Grant number 01M01021) and by FWO Senior Research Grant G011522N. MR is also supported by EPSRC grant
 EP/V005529/1. 


\end{document}